%% file: paper3.tex
\documentclass[reqno]{amsart}
\usepackage{amsaddr}
\usepackage[utf8]{inputenc}
\usepackage[T1]{fontenc}
\usepackage{amsmath, amsthm, amssymb, xcolor, graphicx, lineno}
\usepackage[colorlinks=true, allcolors=blue]{hyperref}

\newtheorem{theorem}{Theorem}
\newtheorem{lemma}{Lemma}
\newtheorem{corollary}{Corollary}

\newtheorem{definition}{Definition}

\newcommand{\norm}[1]{\left\lVert#1\right\rVert}

\title[DNWR for heterogeneous heat]{Dirichlet--Neumann waveform relaxation for heterogeneous heat equations: continuous and time-discrete $L^2$ analysis}
\author[P. Birken, M.J. Gander, N. Kotarsky and L.D.~Lu]{Philipp Birken$^{{\lowercase\mathrm{a}}}$, Martin J. Gander$^{{\lowercase\mathrm{b}}}$, Niklas Kotarsky$^{{\lowercase\mathrm{a}}}$, Liu-Di Lu$^{{\lowercase\mathrm{a}}}$}

\address[Lund]{
	$\,^{\lowercase\mathrm{a}}$ Centre for Mathematical Sciences, 
	Lund University\\ 
	Box 118, 22100 Lund,
	Sweden}

\address[Geneva]{
	$\,^{\lowercase\mathrm{b}}$ Section of Mathematics, 
	University of Geneva\\ 
	Rue de Conseil-General, 7-9 Geneva,
	Switzerland}

\keywords{Dirichlet--Neumann waveform relaxation, error estimate, heterogeneous heat}

\subjclass[2020]{Primary
	65M55; 
}

\begin{document}
\maketitle

\begin{abstract}
We consider two coupled linear heat equations on different spatial domains that interact through a lower dimensional interface. This models conjugate heat transfer. The problem is solved using Dirichlet--Neumann waveform relaxation. This allows us to couple separate codes for the subproblems, a so-called partitioned approach. Our overall goal is to develop more efficient partitioned methods, and to this end, we want reliable error estimates. 

We use an exponentially weighted Fourier technique to derive new error estimates in $L^2$ for finite time $T$ in both continuous and time-discrete settings. We identify an optimized relaxation parameter that guarantees superlinear convergence. Our new continuous estimate predicts linear convergence when $k^2/T$ is small, and superlinear when $k^2/T$ is large, with $k$ the iteration index. Our new time-discrete estimate closely mirrors its continuous counterpart for small $\Delta t$. Otherwise, we obtain linear convergence for large $k$. We also show that convergence is fast when the contrast is large, provided that the small physical parameter domain (e.g. air) is using the Dirichlet transmission condition, and the large physical parameter domain (e.g. steel) is using the Neumann transmission condition in the Dirichlet-Neumann waveform relaxation method. Numerical experiments confirm these findings, and we demonstrate the relevance on the analysis for a nonlinear conjugate heat transfer test case. 
\end{abstract}

	\input{introduction.tex}
	\input{cont2norm.tex}
	\input{timeDisc.tex}	
	\input{compare.tex}
	\input{experiments.tex}

\section{Conclusion}
We derived an optimized relaxation parameter and best assignment of subdomains for fast convergence of the Dirichlet--Neumann waveform relaxation method applied to two coupled heat equations in both continuous and time-discrete cases. 
More precisely, the Dirichlet subdomain must be assigned to the small physical parameter domain (air), and the Neumann condition must be assigned to the large physical parameter domain (steel).
Using the exponentially weighted Fourier technique~\cite{Kwok2025}, we obtain error estimates in $L^2$ in continuous and time-discrete settings that are accurate even in the difficult, but very relevant, nonasymptotic regime of the first few iterations. 
Our new continuous estimate predicts linear convergence when $k^2/T$ is small, and superlinear when $k^2/T$ is large, with $k$ the iteration index. Our new time-discrete estimate closely mirrors its continuous counterpart for small $\Delta t$. Otherwise, we obtain linear convergence for large $k$. We then show for a realistic nonlinear thermal FSI test case in 2D that our error estimates accurately predict the convergence behaviour.

\bibliographystyle{siam}
\bibliography{ourbibliography}

\end{document}

%% file: introduction.tex
\section{Introduction}

Multiphysics problems involve interacting subsystems on different spatial domains, which are coupled through an interface. They are commonly modelled as coupled time-dependent partial differential equations (PDEs) that interact through a shared lower-dimensional interface~\cite{keyes2013}. Applications include fluid-structure interaction~\cite{lorentzon2020}, heat transfer in composite materials~\cite{kowollik2013}, and atmosphere-ocean coupling in climate models~\cite{schuller2025}. Multiphysics problems are often addressed with partitioned approaches that reuse existing single physics solvers on each subdomain. These solvers communicate only through boundary conditions, enabling independent development and discretizations~\cite{rodenberg2025}. Nevertheless, designing stable and efficient coupling algorithms, especially for time-dependent problems, remains challenging.

Waveform relaxation, originally introduced in~\cite{lelarasmee1982} for parallel simulation of large scale electronic circuits, provides a principal framework for coupling time-dependent PDEs: The coupled system is solved iteratively over a time interval by alternatingly solving each subsystem and exchanging time-dependent interface data after each iteration. When combined with domain decomposition techniques, this approach leads to several algorithms, including Schwarz waveform relaxation (SWR), Dirichlet--Neumann waveform relaxation (DNWR), and Neumann-Neumann waveform relaxation (NNWR); see~\cite[Chapter 3]{waveformBook} for a general overview. Applications of waveform relaxation to multiphysics problems can be found in~\cite{clement2022, meimonbir:23, monbir:18, rodenberg2025}. 

Achieving fast convergence in waveform relaxation methods often requires an appropriate choice of relaxation parameters. Suitable choices may lead to {\sl superlinear convergence} on finite time windows. In long time simulations, the global time horizon is typically decomposed into smaller time windows, and waveform relaxation iterations are performed within each window until a stopping criterion is satisfied before marching to the next time window. Therefore, understanding how convergence depends on relaxation parameters, physical parameters, and the length of the time window, is a central topic in the analysis of waveform relaxation methods.

Several convergence analysis techniques for waveform relaxation methods have been developed \cite{arnoult2023, clement2022, engstrom2024, gander2007, haynes2020, ledebl:13, monbir:18}. In the {\it continuous setting}, the authors in~\cite{ledebl:13} derive linear $L^2$ error estimates for SWR applied to heterogeneous advection--diffusion equations on unbounded space-time domains, using Fourier analysis and the Plancherel theorem. In~\cite{gakwma:16}, the authors investigate DNWR and NNWR methods applied to homogeneous heat equations for finite time windows. Based on the Laplace transform and kernel estimates, they derive superlinear $L^{\infty}$ error estimates using special relaxation parameters, $\theta=1/2$ for DNWR and $\theta=1/4$ for NNWR. In~\cite{DD28}, we extend the analysis of DNWR to heterogeneous heat equations, and also derive superlinear $L^{\infty}$ error estimates for a material-dependent relaxation parameter. 

In the {\it time-discrete setting}, the authors in~\cite{hencha:09} provide an analysis for DNWR using one implicit Euler step applied to heterogeneous heat equations. They derive linear $L^2$ error estimates via spatial Fourier transform. An extension to multiple time steps and heterogeneous advection-diffusion equations is presented in~\cite{clement2022}, where the authors apply both implicit Euler and Runge-Kutta schemes, and derive linear $l^2$ error estimates using the $z$-transform on unbounded space-time domains. 

Despite this body of work, to our knowledge, there are no continuous and time-discrete estimates for arbitrary relaxation parameters that depend on the time length $T$. Furthermore, a precise understanding of the role of continuous vs time-discrete, and asymptotic vs nonasymptotic linear respective superlinear convergence behavior, as well as the influence of time step size, time window length and material parameters, is still lacking. 

The present paper closes these gaps. We consider DNWR applied to two coupled linear heat equations with different material parameters. To analyze convergence of DNWR on finite time windows, we introduce a framework based on the exponentially weighted Fourier transform proposed in~\cite{Kwok2025}. Using this, we derive $L^2$ estimates for finite time lengths $T$ in both continuous and time-discrete cases for arbitrary relaxation parameters $\theta$. We also show that the continuous estimate converges superlinearly for a specific choice of $\theta$, and compare both continuous and time-discrete error estimates for this relaxation parameter. 
Our time-discrete analysis reveals that, to preserve superlinear convergence, the time step size $\Delta t$ must be reduced as $T$ decreases. These findings further show the importance of time-discrete estimates, especially when the total simulation time is split into multiple smaller time windows.

The rest of the paper is organized as follows. In Section~\ref{sec:2}, we present a convergence analysis of DNWR applied to two coupled heat equations in the continuous setting and derive superlinear error estimates using the relaxation parameter given in~\cite{DD28}. In Section~\ref{sec:3}, we extend our analysis to the time-discrete setting using the implicit Euler scheme and derive corresponding superlinear error estimates. Section~\ref{sec:4} compares the continuous and time-discrete error estimates. Finally, in Section~\ref{sec:5}, we validate the theoretical results with linear and nonlinear numerical experiments.

%% file: cont2norm.tex
\section{Continuous convergence analysis}\label{sec:2}
	
We consider as our model problem two coupled linear heat equations in one dimension. Let $Q_1 := (-a, 0) \times (0, T)$ and $Q_2 := (0, b) \times (0, T)$ denote two space-time domains with $a>0$, $b>0$, and the time $T>0$. In each domain $Q_j$, we denote by $u_j$ the temperature, $f_j$ the external force, $u_{0, j}$ the initial condition, $\alpha_j$ the volumetric heat capacity, and $\lambda_i$ the heat conductivity. For each domain, the associated heat equation reads
\begin{equation}\label{eq:CoupledSystem}
	\begin{aligned}
	\alpha_1 \partial_t u_1 -\lambda_1 \partial_{xx} u_1 &= f_1& &\text{in } Q_1,&
	\alpha_2 \partial_t u_2 -\lambda_2 \partial_{xx} u_2 &= f_2&  &\text{in } Q_2,\\
	u_1(-a, t) &= g_l(t)& &\text{in } (0,T),&	
	u_1(b, t) &= g_r(t)&  &\text{in } (0,T),\\
	u_1(x, 0) &= u_{0, 1}(x)& &\text{in } (-a,0),& 
	u_2(x, 0) &= u_{0, 2}(x)& &\text{in } (0,b),
\end{aligned}
\end{equation}
where $g_l$ and $g_r$ are two Dirichlet boundary conditions. These two heat equations are coupled at the interface using the two coupling conditions
\begin{equation}\label{eq:CoupleCondition}
u_1(0, t) = u_2(0, t) \ \text{ in } (0,T), \quad 
-\lambda_1 \partial_x u_1(0, t) = -\lambda_2 \partial_x u_2(0, t) \ \text{ in } (0,T).
\end{equation}
The first condition imposes continuity of temperature, whereas the second ensures continuity of the heat flux at the interface.

The goal of our study is to analyze convergence of the DNWR method applied to our model problem~\eqref{eq:CoupledSystem}-\eqref{eq:CoupleCondition}. For any given initial guess $h^0(t)$ and the iteration index $k = 1, 2, \ldots$, the DNWR iteration reads
\begin{equation}\label{eq:DNWR}
\begin{aligned}
	\alpha_1 \partial_t u_1^{k+1} - \lambda_1 \partial_{xx} u_1^{k+1} &= f_1 \ \text{ in } Q_1,&
	\alpha_2 \partial_t u_2^{k+1} - \lambda_2 \partial_{xx} u_2^{k+1} &= f_2 \ \text{ in } Q_2, \\
	u_1^{k+1}(-a, t) &= g_l(t), \ &
	-\lambda_1 \partial_x u_1^{k+1}(0, t) &= -\lambda_2 \partial_x u_2^{k+1}(0, t), \\
	u_1^{k+1}(0, t) &= h^{k}(t), &
	u_1^k(b, t) &= g_r(t), \\
	u_1^{k+1}(x, 0) &= u_{0, 1}(x), &
	u_2^{k+1}(x, 0) &= u_{0, 2}(x),
\end{aligned}
\end{equation}
and we update the Dirichlet transmission condition $h^k$ by
\[h^{k+1}(t) := (1-\theta) h^k(t) + \theta u^{k+1}_2(0, t),
\quad
\theta \in (0, 1),\]
where $\theta$ is the relaxation parameter. Note that here we choose to impose the Dirichlet condition at the interface in $Q_1$ and the Neumann condition at the interface in $Q_2$. 

To present our convergence analysis, we first introduce a framework based on the exponentially weighted Fourier transform proposed in~\cite{Kwok2025}.

\begin{definition}\label{def:FourierTransform}
Let $r>0$, $\omega\in\mathbb{R}$, and $\phi$ a given function. The exponentially weighted Fourier transform of the function $\phi$ is defined by
\[\hat{\phi}(r + i\omega) :=  \int_{-\infty}^{\infty} \phi(t) e^{-(r+i \omega) t} \, dt,\] 
if this integral is finite.
\end{definition}
Note that this transform is the standard Fourier transform applied to $e^{-rt}\phi(t)$, with $r>0$ the weight parameter. We use the hat notation to denote such a transform, and write the argument as $r+i\omega$ to emphasize the dependence on both weight $r$ and frequency $\omega$. From Definition~\ref{def:FourierTransform}, we can verify that the exponentially weighted Fourier transform inherits most of the properties of the Fourier transform, such as linearity, time shifting, and differentiation. Also, it exists if $\phi\in L^1$. 
The Fourier transform of an $L^2$ function is in general defined via extension using the Plancherel theorem, see e.g.~\cite{grafakos2014a}. For the exponentially weighted Fourier transform of an $L^2$ function, we need additional assumptions, since the weight $e^{-rt}$ blows up as $t\to-\infty$ for $r>0$. The next result establishes the existence of this transform under an assumption relevant to our study.

\begin{lemma}\label{prepCont2norm}
Let $r > 0$ and $\phi$ be a function defined in $\mathbb{R}$ with $\norm{\phi}_{L^2(\mathbb{R^+})} < \infty$ and $\phi(t) = 0$ for $t < 0$. Then the exponentially weighted Fourier transform of $\phi$ exists, and we have
\[\frac{1}{2 \pi} \int_{-\infty}^{\infty} |\hat{\phi}(r+i\omega)|^2 d\omega \leq \int_{-\infty}^{\infty} |\phi(t)|^2 dt.\]
\end{lemma}

\begin{proof}
For $r > 0$ and $\phi(t) = 0$ for $t < 0$, we have
\[\int_{-\infty}^{\infty} \big(e^{-rt} \phi(t)\big)^2 \, dt =  \int_{0}^{\infty} \big(e^{-rt}\phi(t)\big)^2 \, dt \leq \int_{0}^{\infty} \big(\phi(t)\big)^2 \, dt < \infty.\]
Thus, $\psi(t) := e^{-r t} \phi(t)$ is in $L^2(\mathbb{R})$. Using the result in~\cite{grafakos2014a}, the Fourier transform of $\psi$ is well defined, meaning that the exponentially weighted Fourier transform of $\phi$ is also well defined. 
Applying now the Plancherel theorem to $\psi$, we obtain 
\[\frac{1}{2 \pi} \int_{-\infty}^{\infty} |\hat{\phi}(r+i\omega)|^2 \, d\omega = \int_{-\infty}^{\infty} \big(e^{-rt} \phi(t)\big)^2 \, dt \leq \int_{0}^{\infty} \big(\phi(t)\big)^2 \, dt = \int_{-\infty}^{\infty} \big(\phi(t)\big)^2 \, dt.\]
\end{proof}

To analyze convergence, we apply the exponentially weighted Fourier transform to the error equation associated with~\eqref{eq:DNWR} and derive the contraction factor. The error equation is obtained by subtracting the solution from the model problem~\eqref{eq:CoupledSystem}-\eqref{eq:CoupleCondition}, or equivalently, by setting directly $f_1$, $f_2$, $u_{0,1}$, $u_{0,2}$, $g_l$, and $g_r$ all to zero in~\eqref{eq:DNWR}. Before applying the exponentially weighted Fourier transform to~\eqref{eq:DNWR}, we also need to extend the initial guess $h^0$ to the time interval $(-\infty, \infty)$ by setting it to zero outside its original domain. The iterates $(u_1^k, u_2^k, h^k)$ for $k>0$ are then extended to $(-\infty,\infty)$ through the iteration itself. For brevity, we retain the notation $h^0$ and $(u_1^k, u_2^k, h^k)$ for the extended functions.

For $\omega\in\mathbb{R}$ and $r\in\mathbb{R}^+$, the exponentially weighted Fourier transform applied to the error equation associated with~\eqref{eq:DNWR} gives
\begin{equation}\label{eq:transformedDNWR}
\begin{aligned}
	\alpha_1 (r+i \omega) \hat{u}_1^{k+1} -\lambda_1 \partial_{xx} \hat{u}_1^{k+1} &= 0 \quad \text{ in } (-a,0),\\
	\hat{u}_1^{k+1}(-a, r+i \omega) &= 0,\\
	\hat{u}_1^{k+1}(0, r+i\omega) &= \hat h^{k}(r+i \omega),\\
	\alpha_2 (r+i \omega) \hat{u}_2^{k+1} -\lambda_2 \partial_{xx} \hat{u}_2^{k+1} &= 0 \quad \text{ in } (0,b),\\
	-\lambda_1 \partial_x \hat{u}_1^{k+1}(0, r+i \omega) &= -\lambda_2 \partial_x \hat{u}_2^{k+1}(0, r+i \omega),\\
	u_1^k(b, r+i \omega) &= 0, \\
	\hat{h}^{k+1}(r+i\omega) = (1-\theta) \hat{h}^k&(r+i\omega) + \theta \hat{u}^{k+1}_2(0, r+i \omega).
\end{aligned}
\end{equation}
These are two coupled second-order ordinary differential equations in space, and solving them for given $\omega$ and $r$ yields the closed form solutions
\small{\[\begin{aligned}
	\hat{u}_1^{k+1}(x, r+i \omega) &= \frac{\sinh\left( \sqrt{\frac{\alpha_1 (r+i \omega)}{\lambda_1}} (x +a) \right)}{\sinh\left( \sqrt{\frac{\alpha_1 (r+i \omega)}{\lambda_1}} a \right)} \hat{h}^k(r+i \omega), \\
	\hat{u}_2^{k+1}(x, r+i \omega) &= 
	\sqrt{\frac{\alpha_1 \lambda_1}{\alpha_2 \lambda_2}}  \coth\left( \sqrt{\frac{\alpha_1 (r+i \omega)}{\lambda_1}} a\right)  \frac{\sinh\left( \sqrt{\frac{\alpha_2 (r+i \omega)}{\lambda_2}} (x - b) \right)}{\cosh\left( \sqrt{\frac{\alpha_2 (r+i \omega)}{\lambda_2}} b \right)} \hat{h}^k(r+i \omega).
\end{aligned}\]}
Using the last equation in~\eqref{eq:transformedDNWR}, we obtain $\hat{h}^k(r+i \omega) = \hat{\rho}(r+i \omega)^k \hat{h}^0(r+i \omega)$, where $\hat{\rho}$ is given by
\begin{equation}\label{eq:cvfCont}
\hat{\rho}(s) := 1 - \theta - \theta \sqrt{\frac{\alpha_1 \lambda_1}{\alpha_2 \lambda_2}}\coth\left(\sqrt{\frac{\alpha_1 s}{\lambda_1}} a\right) \tanh\left(\sqrt{\frac{\alpha_2 s}{\lambda_2}} b\right), \quad s\in\mathbb{C}.
\end{equation}
The function $\hat{\rho}$ evaluated at $s=r+i\omega$ with $\text{Re}(s) = r>0$ is the contraction factor associated with the transformed DNWR iteration~\eqref{eq:transformedDNWR}. Note that $ \hat{\rho}(r+i \omega)$ depends on the domain size $a$ and $b$, and the physical parameters $\alpha_j$ and $\lambda_j$, $j=1, 2$. 
Regarding the relaxation parameter $\theta$, one wants the associated contraction factor to be small, such that ~\eqref{eq:DNWR} converges very fast. Before discussing the choice of $\theta$, we first present an error estimate for arbitrary relaxation parameter.

\begin{theorem}\label{2normSuperCont}
	For any given $\theta\in (0, 1)$ and $h^0\in L^2(0, T)$, the error of the iteration~\eqref{eq:DNWR} in the time interval $(0, T)$ satisfies the estimate 
	\begin{equation}\label{eq:GenEstCon}
	\|h^k\|_{L^2(0, T)} \leq \min_{r > 0} e^{rT}\max_{\omega \in \mathbb{R}} \left|\hat{\rho}(r + i \omega)\right|^k \, \norm{h^0}_{L^2(0, T)}.
	\end{equation}
\end{theorem} 

\begin{proof}
	By \cite[Chapter 4]{Lions72}, $h^k$ exists and is well defined for $h^0 \in L^2(0,T)$. For $r > 0$, we have
	\[\begin{aligned}
		\|h^k\|_{L^2(0, T)}^2 = \int_{0}^{T} \left(h^{k}(t)\right)^2 dt = \int_{0}^{T} e^{2rt} e^{-2 r t} \left(h^{k}(t)\right)^2 &dt \leq e^{2rT} \int_{0}^{T} \left(e^{-r t}h^{k}(t)\right)^2 dt \\
		&\leq e^{2rT} \int_{-\infty}^{\infty} \left( e^{- r t} h^{k}(t)\right)^2 dt,
	\end{aligned}\]
	where in the last inequality we use the extension of $h^k$ to $\mathbb{R}$. By Lemma~\ref{prepCont2norm}, the exponentially weighted Fourier transform of $h^k$ exists. Applying the Plancherel theorem, we find
	\[\begin{aligned}
		e^{2rT} \int_{-\infty}^{\infty} \left( e^{- r t}h^{k}(t)\right)^2 dt 
		&= \frac{e^{2rT}}{2 \pi} \int_{-\infty}^{\infty} \left| \hat{h}^{k}(r + i \omega) \right|^2 d\omega \\
		&= \frac{e^{2rT}}{2 \pi} \int_{-\infty}^{\infty} \left| \hat{\rho}(r + i \omega)^k \hat{h}^0(r + i \omega) \right|^2 d\omega \\
		&\leq \frac{e^{2rT}}{2 \pi} \max_{\omega \in \mathbb{R}} \left|\hat{\rho}(r  + i\omega)\right|^{2k} \int_{-\infty}^{\infty} \left| \hat{h}^{0}(r + i \omega) \right|^2 d\omega \\
		&= e^{2rT} \max_{\omega \in \mathbb{R}} \left|\hat{\rho}(r + i\omega)\right|^{2k}\int_{-\infty}^{\infty} \left| h^{0}(t) \right|^2 dt \\
		&= e^{2rT} \max_{\omega \in \mathbb{R}} \left|\hat{\rho}(r + i\omega)\right|^{2k}\int_{0}^{T} \left| h^{0}(t) \right|^2 dt,
	\end{aligned}\]
	where we use Lemma~\ref{prepCont2norm} once again, and the fact that $h^0$ is extended from $(0, T)$ to $(-\infty, \infty)$ by 0 for the last equality. Combining these computations gives 
	\[\|h^k\|_{L^2(0, T)}^2 \leq e^{2rT} \max_{\omega \in \mathbb{R}} |\hat{\rho}(r + i\omega)|^{2k} \|h^0\|_{L^2(0, T)}^2.\]
	Taking the square root on both sides and minimizing over all $r>0$ leads to~\eqref{eq:GenEstCon}.
\end{proof}

Theorem~\ref{2normSuperCont} provides a general error estimate of~\eqref{eq:DNWR} for any relaxation parameter. To achieve fast convergence, we can choose a relaxation parameter $\theta$ such that the convergence factor~\eqref{eq:cvfCont} is small. In particular, the choice
\begin{equation}\label{eq:thtopt}
\theta = \frac1{1 + \sqrt{\frac{\alpha_1 \lambda_1}{\alpha_2 \lambda_2}}\coth\left(\sqrt{\frac{\alpha_1 (r+i \omega)}{\lambda_1}} a\right) \tanh\left(\sqrt{\frac{\alpha_2 (r+i \omega)}{\lambda_2}} b\right)}
\end{equation}
yields convergence in two iterations. However, this choice depends both on the weight parameter $r$ and the Fourier frequency $\omega$, making it not useful in practice. On the other hand, when the physics in both subdomains satisfy $\sqrt{\alpha_1/\lambda_1} a = \sqrt{\alpha_2/\lambda_2} b$, the contraction factor~\eqref{eq:cvfCont} reduces to the simple form $\hat{\rho} = 1 - \theta - \theta\sqrt{\alpha_1 \lambda_1}/\sqrt{\alpha_2 \lambda_2}$, which is independent of the $\omega$ and $r$. In this case, we obtain the optimal relaxation parameter
\begin{equation}\label{eq:thetaopt}
	\theta^* := \frac{\sqrt{\alpha_2\lambda_2}}{\sqrt{\alpha_1\lambda_1}+\sqrt{\alpha_2\lambda_2}}.
\end{equation}
Note that the choice $\theta$ given in~\eqref{eq:thtopt} converges to $\theta^*$ when $\omega$ goes to infinity.
Therefore, for high frequencies, the DNWR iteration always converges very fast for $\theta=\theta^*$, and thus convergence will be independent of the mesh size, as discussed in~\cite[Chapter 4.7, p.174]{Ciaramella2022}. 
Furthermore, the parameter $\theta^*$ is much more convenient to use in practice, as it only depends on the physics of the problem. Therefore, we are interested in the impact of $\theta^*$ on the convergence of~\eqref{eq:DNWR} in more general cases, {\it i.e.}, $\sqrt{\alpha_1/\lambda_1} a \neq \sqrt{\alpha_2/\lambda_2} b$. Substituting $\theta^*$ into ~\eqref{eq:cvfCont} and using properties of the differences of arguments for hyperbolic functions, we obtain
\begin{equation}\label{eq:rhotheta}
\hat{\rho}(s)|_{\theta^*} = -\frac{1}{1+\sqrt{\frac{\alpha_2\lambda_2}{\alpha_1\lambda_1}}}\frac{\sinh\left(\left(\sqrt{\frac{\alpha_1}{\lambda_1}} a - \sqrt{\frac{\alpha_2}{\lambda_2}} b\right) \sqrt{s}\right)}{\sinh\left(\sqrt{\frac{\alpha_1}{\lambda_1}} a\sqrt{s}\right)\cosh\left(\sqrt{\frac{\alpha_2}{\lambda_2}} b\sqrt{s}\right)},
\quad 
s \in \mathbb{C}.
\end{equation}
To simplify the notation, we also introduce the function \begin{equation}\label{def:H}
	\hat H(s) := \frac{\sinh\left(|\alpha-\beta|\sqrt{s}\right)}{\sinh\left(\alpha\sqrt{s}\right)\cosh\left(\beta\sqrt{s}\right)}, \quad s \in \mathbb{C}.
\end{equation}
With the help of~\eqref{def:H}, the error estimate~\eqref{eq:GenEstCon} associated with $\theta = \theta^*$ is then
\begin{equation}\label{eq:SpeEstCon}
\|h^k\|_{L^2(0, T)} \leq \left(\frac{1}{1+\sqrt{\frac{\alpha_2\lambda_2}{\alpha_1\lambda_1}}}\right)^k  \min_{r > 0} e^{rT}\max_{\omega \in \mathbb{R}} \left|\hat H(r + i \omega)\right|^k \, \norm{h^0}_{L^2(0, T)},
\end{equation}
where $\alpha = \sqrt{\alpha_1/\lambda_1} a$ and $\beta = \sqrt{\alpha_2/\lambda_2} b$ in $\hat H$ for \eqref{def:H}. Therefore, it is sufficient to study the min-max problem associated with $\hat H$. We first solve the maximization problem with respect to the frequency $\omega$, then use $r$ as an auxiliary parameter to find a superlinear estimate. The next result shows the global maximum of the maximization problem.

\begin{lemma}\label{lem:max}
	Let $\alpha, \beta, r>0$ be given. Then the function $\hat H$ defined in~\eqref{def:H} satisfies
	$\max_{\omega \in \mathbb{R}}|\hat H(r + i\omega)| = \hat H(r)$.
\end{lemma}
\begin{proof}
	In the proof of~\cite[Theorem 2]{ma:13}, it is shown that the function $\hat H(s)$ with $\alpha >0$, $\beta > 0$ and $\text{Re}(s)>0$, is the Laplace transform of a positive function denoted by $H(t)$. Thus, we have
\[\begin{aligned}
\max_{\omega \in \mathbb{R}} |\hat{H}(r + i \omega)| =  \max_{\omega \in \mathbb{R}} \Big| \int_{-\infty}^{\infty} &H(t) e^{-r t - i \omega t} dt \Big|
\leq \max_{\omega \in \mathbb{R}} \int_{-\infty}^{\infty} |H(t)|e^{-rt} dt  \\
&= \int_{-\infty}^{\infty} H(t)e^{-rt} dt  = \hat{H}(r).
\end{aligned}\] 
To see that this is actually an equality, it suffices to substitute $\omega = 0$ into $\hat{H}(r + i \omega)$, which gives $\hat{H}(r)$, and thus the proof is complete.
\end{proof}

Applying Lemma~\ref{lem:max} to \eqref{eq:SpeEstCon}, we have
\begin{equation}\label{eq:min-cont}
\norm{h^k}_{L^2(0, T)} \leq \left(\frac{1}{1+\sqrt{\frac{\alpha_2\lambda_2}{\alpha_1\lambda_1}}}\right)^k\min_{r>0}e^{rT} \left(\hat H(r)\right)^k \norm{h^0}_{L^2(0, T)},
\end{equation}
It remains to treat the minimization problem related to $r$. Depending on the physics of the problem, we have two results.

\begin{theorem}[Case $\sqrt{\alpha_1/\lambda_1} a > \sqrt{\alpha_2/\lambda_2} b$]\label{thm:cont-case1}
If $\theta = \theta^*$ and $\sqrt{\alpha_1/\lambda_1} a > \sqrt{\alpha_2/\lambda_2} b$, then the error of the algorithm~\eqref{eq:DNWR} satisfies 
\begin{equation}\label{eq:ContEst1}
\norm{h^k}_{L^2(0, T)} \leq \left(2\frac{1 - \sqrt{\frac{\alpha_2 \lambda_1}{\alpha_1\lambda_2}}\frac{b}{a}}{1 + \sqrt{\frac{\alpha_2 \lambda_2}{\alpha_1\lambda_1}}} \right)^k e^{-\frac{\alpha_2 k^2b^2}{4\lambda_2 T}} \norm{h^0}_{L^2(0, T)}.
\end{equation}
\end{theorem}

\begin{proof}
The idea is to first bound the function $\hat H(r)$ by an exponential function, and then to analyze the exponential function to find a minimum.
Let $0<\alpha<\beta$ and $f(x) = \beta\sinh(\alpha x) - \alpha \sinh(\beta x)$. We have $f'(x) = \alpha\beta(\cosh(\alpha x) - \cosh(\beta x)) < 0$, i.e., $f$ is decreasing. Hence $f(x) \leq f(0) = 0$, $\forall x>0$, yielding $\sinh(\alpha x) / \sinh(\beta x) \leq \alpha / \beta$. Taking $\alpha = \sqrt{\alpha_1/\lambda_1} a - \sqrt{\alpha_2/\lambda_2} b$, $\beta = \sqrt{\alpha_1/\lambda_1} a$ and $x = \sqrt{r}$, we find
\begin{equation}\label{eq:sinhbound}
\frac{\sinh\left((\sqrt{\frac{\alpha_1}{\lambda_1}} a - \sqrt{\frac{\alpha_2}{\lambda_2}} b) \sqrt{r} \right)}{\sinh\left(\sqrt{\frac{\alpha_1}{\lambda_1}} a \sqrt{r}\right)}\leq \frac{\sqrt{\frac{\alpha_1}{\lambda_1}} a - \sqrt{\frac{\alpha_2}{\lambda_2}} b}{\sqrt{\frac{\alpha_1}{\lambda_1}} a} = 1 - \sqrt{\frac{\alpha_2 \lambda_1}{\alpha_1\lambda_2}}\frac{b}{a}.
\end{equation}
On the other hand, using the definition of the hyperbolic cosine, we have
\begin{equation}\label{eq:coshbound}
\frac{1}{\cosh\left(\sqrt{\frac{\alpha_2}{\lambda_2}} b \sqrt{r} \right)} 
= \frac{2 \exp\left(-\sqrt{\frac{\alpha_2}{\lambda_2}} b \sqrt{r}\right)}{1 + \exp \left(-2\sqrt{\frac{\alpha_2}{\lambda_2}} b \sqrt{r} \right)} 
\leq 2 \exp\left(-\sqrt{\frac{\alpha_2}{\lambda_2}} b\sqrt{r}\right).
\end{equation}
The product of~\eqref{eq:sinhbound} and~\eqref{eq:coshbound} to the power $k$ gives
\begin{equation}\label{eq:BoundThm2}
\left(\frac{\sinh\left(\left(\sqrt{\frac{\alpha_1}{\lambda_1}} a - \sqrt{\frac{\alpha_2}{\lambda_2}} b\right) \sqrt{r} \right)}{\sinh\left(\sqrt{\frac{\alpha_1}{\lambda_1}} a \sqrt{r}\right)\cosh\left(\sqrt{\frac{\alpha_2}{\lambda_2}} b \sqrt{r}\right)}\right)^k
\leq 2^k\left(1 - \sqrt{\frac{\alpha_2 \lambda_1}{\alpha_1\lambda_2}}\frac{b}{a}\right)^k \exp\left(-kb\sqrt{\frac{\alpha_2r}{\lambda_2}}\right).
\end{equation}
Therefore, the estimate~\eqref{eq:min-cont} with the definition~\eqref{def:H} of $\hat H$ becomes
\begin{equation}\label{eq:minBoundThm2}
\begin{aligned}
\norm{h^k}_{L^2(0, T)} \leq& \left(\frac{1}{1+\sqrt{\frac{\alpha_2\lambda_2}{\alpha_1\lambda_1}}}\right)^k\min_{r>0}e^{rT} \left(\hat H(r)\right)^k \norm{h^0}_{L^2(0, T)}
\\
\leq& \left(2\frac{1 - \sqrt{\frac{\alpha_2 \lambda_1}{\alpha_1\lambda_2}}\frac{b}{a}}{1 + \sqrt{\frac{\alpha_2 \lambda_2}{\alpha_1\lambda_1}}} \right)^k \min_{r>0}e^{rT - kb\sqrt{\frac{\alpha_2r}{\lambda_2}}}\norm{h^0}_{L^2(0, T)}.
\end{aligned}
\end{equation}

It remains now to find the minimum of $\exp(rT - kb\sqrt{\alpha_2r}/\sqrt{\lambda_2})$. Differentiating it with respect to $r$ gives $(T - k b\sqrt{\alpha_2}/\sqrt{4\lambda_2 r}) \exp(rT - kb \sqrt{\alpha_2 r}/\sqrt{\lambda_2})$. Setting the derivative to 0 yields the extremal point
\begin{equation}\label{eq:ropt-case1}
r^*_{1c} = \frac{\alpha_2k^2 b^2}{4\lambda_2T^2}.
\end{equation}
Note that the derivative changes sign from negative to positive at $r^*_{1c}$, and thus $r^*_{1c}$ is a minimum. Substituting $r^*_{1c}$ into the second inequality in~\eqref{eq:minBoundThm2} gives the estimate~\eqref{eq:ContEst1}.
\end{proof}

\begin{theorem}[Case $\sqrt{\alpha_1/\lambda_1} a < \sqrt{\alpha_2/\lambda_2} b$]\label{thm:cont-case2}
Let $\theta = \theta^*$. If $\sqrt{\alpha_1/\lambda_1} a < \sqrt{\alpha_2/\lambda_2} b < 2\sqrt{\alpha_1/\lambda_1} a$, then the error of the algorithm~\eqref{eq:DNWR} satisfies 
\begin{equation}\label{eq:estimate-case2-1}
\norm{h^k}_{L^2(0, T)} \leq \left(2\frac{\sqrt{\frac{\alpha_2 \lambda_1}{\alpha_1\lambda_2}}\frac{b}{a} - 1}{\sqrt{\frac{\alpha_2 \lambda_2}{\alpha_1\lambda_1}}+1} \right)^k e^{-\frac{\alpha_2 k^2b^2}{4\lambda_2 T}} \norm{h^0}_{L^2(0, T)}.
\end{equation}
If $\sqrt{\alpha_2/\lambda_2} b \geq 2\sqrt{\alpha_1/\lambda_1} a$, then the error of the algorithm~\eqref{eq:DNWR} satisfies 
\begin{equation}\label{eq:estimate-case2-2}
\norm{h^k}_{L^2(0, T)} \leq \left(2\frac{\sqrt{\frac{\alpha_2 \lambda_1}{\alpha_1\lambda_2}}\frac{b}{a} - 1}{\sqrt{\frac{\alpha_2 \lambda_2}{\alpha_1\lambda_1}}+1} \right)^k e^{-\frac{\alpha_1 k^2a^2}{\lambda_1 T}} \norm{h^0}_{L^2(0, T)}.
\end{equation}
\end{theorem}

\begin{proof}
The idea is the same as in the proof of Theorem~\ref{thm:cont-case1}.
If $\sqrt{\alpha_1/\lambda_1} a < \sqrt{\alpha_2/\lambda_2} b < 2\sqrt{\alpha_1/\lambda_1} a$, then $0<\sqrt{\alpha_2/\lambda_2} b - \sqrt{\alpha_1/\lambda_1} a < \sqrt{\alpha_1/\lambda_1} a$. We can thus follow the same steps as in the proof of Theorem~\ref{thm:cont-case1} by taking $\alpha = \sqrt{\alpha_2/\lambda_2} b - \sqrt{\alpha_1/\lambda_1} a$ and $\beta = \sqrt{\alpha_1/\lambda_1} a$.
We find a similar bound as~\eqref{eq:BoundThm2}, namely
\begin{equation}\label{eq:BoundThm3}
	\left(\frac{\sinh\left(\left(\sqrt{\frac{\alpha_2}{\lambda_2}} b - \sqrt{\frac{\alpha_1}{\lambda_1}} a \right) \sqrt{r} \right)}{\sinh\left(\sqrt{\frac{\alpha_1}{\lambda_1}} a \sqrt{r}\right)\cosh\left(\sqrt{\frac{\alpha_2}{\lambda_2}} b \sqrt{r}\right)}\right)^k
	\leq 2^k\left(\sqrt{\frac{\alpha_2 \lambda_1}{\alpha_1\lambda_2}}\frac{b}{a} - 1\right)^k \exp\left(-kb\sqrt{\frac{\alpha_2r}{\lambda_2}}\right).
\end{equation}
This leads to the same miminizer $r^*_{1c}$ given in~\eqref{eq:ropt-case1}, and we obtain the estimate~\eqref{eq:estimate-case2-1}.

If $\sqrt{\alpha_2/\lambda_2} b \geq 2\sqrt{\alpha_1/\lambda_1} a$, we substitute $\tau=\exp(-2\sqrt{r})$ and obtain 
\[\frac{\sinh\left(\left(\sqrt{\frac{\alpha_2}{\lambda_2}} b - \sqrt{\frac{\alpha_1}{\lambda_1}} a \right) \sqrt{r} \right)}{\sinh\left(\sqrt{\frac{\alpha_1}{\lambda_1}} a \sqrt{r}\right)\cosh\left(\sqrt{\frac{\alpha_2}{\lambda_2}} b \sqrt{r}\right)}
= 2 \tau^{\sqrt{\frac{\alpha_1}{\lambda_1}} a}\frac{1-\tau^{\sqrt{\frac{\alpha_2}{\lambda_2}} b-\sqrt{\frac{\alpha_1}{\lambda_1}} a}}{\left(1+\tau^{\sqrt{\frac{\alpha_2}{\lambda_2}} b}\right)\left(1-\tau^{\sqrt{\frac{\alpha_1}{\lambda_1}} a}\right)}.\] 
The latter fraction can be bounded by
\[\frac{1-\tau^{\sqrt{\frac{\alpha_2}{\lambda_2}} b-\sqrt{\frac{\alpha_1}{\lambda_1}} a}}{\left(1+\tau^{\sqrt{\frac{\alpha_2}{\lambda_2}} b}\right)\left(1-\tau^{\sqrt{\frac{\alpha_1}{\lambda_1}} a}\right)} 
\leq \frac{1-\tau^{\sqrt{\frac{\alpha_2}{\lambda_2}} b-\sqrt{\frac{\alpha_1}{\lambda_1}} a}}{1-\tau^{\sqrt{\frac{\alpha_1}{\lambda_1}} a}}.\] 
Let now $\alpha \geq \beta >0$ and $f(x) = \beta(1 - e^{-\alpha x}) - \alpha(1 - e^{-\beta x})$, then $f'(x) = \alpha\beta(e^{-\alpha x} - e^{-\beta x}) \leq 0$, i.e., $f$ is non-increasing. Hence, $f(x) \leq f(0) = 0$, $\forall x > 0$, yielding $(1 - e^{-\alpha x})/((1 - e^{-\beta x})) \leq \alpha/\beta$. Taking then $\alpha = \sqrt{\alpha_2/\lambda_2} b - \sqrt{\alpha_1/\lambda_1} a$, $\beta = \sqrt{\alpha_1/\lambda_1} a$ and $x = 2\sqrt{r}$ gives
\begin{equation}\label{eq:bound-exp}
	\frac{1-\tau^{\sqrt{\frac{\alpha_2}{\lambda_2}} b-\sqrt{\frac{\alpha_1}{\lambda_1}} a}}{1-\tau^{\sqrt{\frac{\alpha_1}{\lambda_1}} a}} 
	\leq \frac{\sqrt{\frac{\alpha_2}{\lambda_2}} b-\sqrt{\frac{\alpha_1}{\lambda_1}} a}{\sqrt{\frac{\alpha_1}{\lambda_1}} a} = \sqrt{\frac{\alpha_2 \lambda_1}{\alpha_1\lambda_2}}\frac{b}{a} - 1.
\end{equation} 
This implies that 
\begin{equation}\label{eq:BoundThm4}
\left(\frac{\sinh\left(\left(\sqrt{\frac{\alpha_2}{\lambda_2}} b - \sqrt{\frac{\alpha_1}{\lambda_1}} a \right) \sqrt{r} \right)}{\sinh\left(\sqrt{\frac{\alpha_1}{\lambda_1}} a \sqrt{r}\right)\cosh\left(\sqrt{\frac{\alpha_2}{\lambda_2}} b \sqrt{r}\right)}\right)^k
\leq 2^k \left(\sqrt{\frac{\alpha_2 \lambda_1}{\alpha_1\lambda_2}}\frac{b}{a} - 1\right)^k \exp\left(-2ka\sqrt{\frac{\alpha_1r}{\lambda_1}} \right).
\end{equation} 
Note that the exponential part is slightly different compared with~\eqref{eq:BoundThm3} in the previous case. Using the bound~\eqref{eq:BoundThm4}, our estimate~\eqref{eq:min-cont} in this case becomes
\begin{equation}\label{eq:minBoundThm3}
\begin{aligned}
\norm{h^k}_{L^2(0, T)} \leq& \left(\frac{1}{1+\sqrt{\frac{\alpha_2\lambda_2}{\alpha_1\lambda_1}}}\right)^k\min_{r>0}e^{rT} \left(\hat H(r)\right)^k \norm{h^0}_{L^2(0, T)}
\\
\leq& \left(2\frac{\sqrt{\frac{\alpha_2 \lambda_1}{\alpha_1\lambda_2}}\frac{b}{a} - 1}{\sqrt{\frac{\alpha_2 \lambda_2}{\alpha_1\lambda_1}}+1}  \right)^k \min_{r>0}e^{rT - 2ka\sqrt{\frac{\alpha_1r}{\lambda_1}}}\norm{h^0}_{L^2(0, T)}.
\end{aligned}
\end{equation}
Once again, we take the derivative with respect to $r$ of the function $\exp(rT - 2ka\sqrt{\alpha_1r}/\sqrt{\lambda_1})$ to find the minimum 
\begin{equation}\label{eq:ropt-case2}
r^*_{2c} = \frac{\alpha_1 k^2a^2}{\lambda_1 T^2}.
\end{equation}
Substituting $r^*_{2c}$ back into the second inequality of~\eqref{eq:minBoundThm3} gives the estimate~\eqref{eq:estimate-case2-2}.
\end{proof}

For $\theta = \theta^*$, the error estimates given in Theorems~\ref{thm:cont-case1} and~\ref{thm:cont-case2} contain both linear and superlinear behavior, which is characterized by the following result.

\begin{corollary}\label{cor:cont}
For given parameters $\alpha_j>0, \lambda_j>0$, $j=1,2$ and $a>0, b >0$, 
\begin{itemize}
\item when $k^2/T$ is small (large $T$ or small iteration count $k$), then the continuous estimates given in Theorems~\ref{thm:cont-case1} and~\ref{thm:cont-case2} show linear convergence;  
\item when $k^2/T$ is large (small $T$ or large $k$), they show superlinear convergence.
\end{itemize}
\end{corollary}

\begin{proof}
For small $k^2/T$, the exponential part in~\eqref{eq:ContEst1}, \eqref{eq:estimate-case2-1} and \eqref{eq:estimate-case2-2} is close to 1, and thus the factor in front gives linear convergence.
For large $k^2/T$, the exponential part behaves like $\exp(-k^2/T)$, and typical superlinear convergence is obtained.
\end{proof}

The weight $r$ is an auxiliary parameter that we only use to derive error estimates. As shown in the proofs of Theorems~\ref{thm:cont-case1} and~\ref{thm:cont-case2}, our approach to treat the minimization problem~\eqref{eq:min-cont} is as follows: we first bound $\hat H(r)$ by a function of $r$, and then solve the resulting minimization problem to obtain $r^*_{1c}$ and $r^*_{2c}$, instead of solving~\eqref{eq:min-cont} directly. We can then substitute $r^*_{1c}$ and $r^*_{2c}$ into the function $\hat H(r)$ in~\eqref{eq:min-cont} to obtain sharper error estimates.

\begin{corollary}\label{thm:min-cont-bis}
For $\theta = \theta^*$, the error of the algorithm~\eqref{eq:DNWR} satisfies 
\begin{equation}\label{eq:min-cont-bis}
\norm{h^k}_{L^2(0, T)} 
\leq \left(\frac{ \hat H(r^*)}{1+\sqrt{\frac{\alpha_2\lambda_2}{\alpha_1\lambda_1}}}\right)^k e^{r^*T}\norm{h^0}_{L^2(0, T)},
\end{equation} 
where $r^*$ denotes $r_{1c}^*$ or $r_{2c}^*$ depending on the value of $\sqrt{\alpha_1/\lambda_1} a$ and $\sqrt{\alpha_2/\lambda_2} b$. 
\end{corollary}

Although the error estimate~\eqref{eq:min-cont-bis} is sharper than those in Theorems~\ref{thm:cont-case1} and~\ref{thm:cont-case2}, it is less straightforward to see the linear and superlinear convergence behavior in~\eqref{eq:min-cont-bis}, as the dependence on $T$ is hidden in the value of $r^*$. 

In general, one splits a long simulation time into smaller time windows, performs waveform relaxation methods within each window until a stopping criterion is satisfied, and then moves to the next one. Therefore, we are interested in the values of $T$ that lead to superlinear convergence. The error estimates provided in Theorems~\ref{thm:cont-case1} and~\ref{thm:cont-case2} indicate how one should choose the time window length, which will be further discussed in Section~\ref{sec:4}.

%% file: timeDisc.tex
\section{Time-discrete analysis}\label{sec:3}

The error estimates obtained in the continuous setting do not clarify the influence of the time step size. In this section, we provide a detailed time-discrete analysis of~\eqref{eq:DNWR}. We discretize the time interval $[0, T]$ with $0=t_0 < t_1 < \ldots < t_N = T$, where $t_n = n\Delta t$ and $\Delta t = T/N$. We denote by $u^n$ the approximation of the solution $u$ at the time $t_n$ and $\boldsymbol{u} = (u^1, \ldots, u^N)^\top$. As before, we use $k$ to denote the index of the waveform iteration. The implicit Euler method applied to~\eqref{eq:DNWR} reads: For $n=0, \ldots, N-1$,
\begin{equation}\label{eq:DNWRIE}
	\begin{aligned}
		\frac{\alpha_1}{\Delta t} (u_1^{n+1, k+1} - u_1^{n, k+1}) - \lambda_1 \partial_{xx} u_1^{n+1, k+1} &= f_1^{n+1} \ \text{ in } \ (-a,0) , \\
		u_1^{n+1, k+1}(-a) = g_l^{n+1}, \quad 
		u_1^{n+1, k+1}(0) &= h^{n+1, k},\\
		u_1^{0, k+1}(x) &= u_{0, 1}(x) \\
		\frac{\alpha_2}{\Delta t} (u_2^{n+1, k+1}-u_2^{n, k+1})-\lambda_2 \partial_{xx} u_2^{n+1, k+1} &= f_2^{n+1} \ \text{ in } \ (0,b),\\
		u_2^{n+1,k+1}(b) = g_r^{n+1}, \quad
		-\lambda_1 \partial_x u_1^{n+1, k+1}(0) &= -\lambda_2 \partial_x u_2^{n+1, k+1}(0), \\
		u_2^{0, k+1}(x) &= u_{0, 2}(x),\\
		h^{n+1, k+1} := (1-\theta) h^{n+1, k} + \theta &u^{n+1, k+1}_2(0),  \quad \theta\in(0, 1).
	\end{aligned}
\end{equation}

To remain in a similar analysis framework as in Section~\ref{sec:2}, we keep using the exponentially weighted technique. In the continuous case, we use the Fourier transform; correspondingly, we consider here the discrete-time Fourier transform, which possesses properties, such as linearity, time shifting, differentiation, and the Parseval Theorem. We refer to the monograph~\cite[Chapter 2.9]{Oppenheim1999} for a detailed introduction. To take advantage of the Parseval Theorem, we define the exponentially weighted discrete-time Fourier transform as follows.

\begin{definition}\label{def:DiscFourier}
Let $r>0$, $\Delta t>0$, $\omega \in [-\pi, \pi]$,  and $\{\phi^n\}_{n\in\mathbb{Z}}$ be a given sequence. 
The exponentially weighted discrete-time Fourier transform of $\{\phi^n\}_{n\in\mathbb{Z}}$ is defined by
\[\widehat{\phi}(r\Delta t + i \omega) := \sum_{n\in\mathbb{Z}} \phi^n e^{-n(r\Delta t+i\omega)},\]
if this series is finite. 
\end{definition}

We use the notation wide hat here to denote the exponentially weighted discrete-time Fourier transform to distinguish it from the exponentially weighted Fourier transform in the continuous case. In addition, similar to Definition~\ref{def:FourierTransform}, we write the argument as $r\Delta t + i\omega$ to emphasize the dependence on the weight parameter $r$, the time step size $\Delta t$ and the frequency $\omega$. 
In general, the weight $e^{-rn\Delta t}$ blows up as $n\to -\infty$ for $r>0$. The next Lemma provides a similar result to Lemma~\ref{prepCont2norm}.

\begin{lemma}\label{prepDisc2norm}
Let $r > 0$ and $\boldsymbol{\phi} = \{\phi^n\}_{n\in\mathbb{Z}}$ be a sequence such that $ \sum_{n\in\mathbb{N}} (\phi^n)^2 < \infty$, and $\phi^n=0$ for $n<0$. Then the exponentially weighted discrete-time Fourier transform exists, and we have
\[\frac{1}{2 \pi} \int_{-\pi}^{\pi} |\widehat{\phi}(r\Delta t+i\omega)|^2 d\omega \leq \norm{\boldsymbol{\phi}}_{l^2}^2.\]
\end{lemma}

\begin{proof}
For $r>0$ and $\Delta t> 0$, $e^{-r \Delta t} < 1$, and we find $\sum_{n\in\mathbb{Z}}(e^{-r \Delta t} \phi^n)^2 \leq \sum_{n\in\mathbb{Z}} (\phi^n)^2 < \infty$. The time-discrete Fourier transform of the sequence $e^{-r \Delta t } \boldsymbol{\phi}$ therefore exists. Applying Parseval's theorem to $e^{-r \Delta t } \boldsymbol{\phi}$, we find
\[\frac{1}{2\pi} \int_{-\pi}^{\pi} \left|\widehat{\phi}(r \Delta t + i\omega)\right|^2 d \omega= \sum_{n\in\mathbb{Z}} (e^{-r \Delta t n} \phi^{n})^2 \leq \sum_{n\in\mathbb{Z}} (\phi^{n})^2 = \norm{\boldsymbol{\phi}}_{l^2}^2.\]
\end{proof}

To analyze convergence in the time-discrete case, we focus on the error equation associated with~\eqref{eq:DNWRIE}, which consists of setting $f_1^n, f_2^n, g_l^n, g_r^n, u_{0, 1}, u_{0, 2}$ all to zero in~\eqref{eq:DNWRIE}. We also extend the initial guess $\{h^{n, 0}\}_{n=0}^{N}$ by 0 to $\boldsymbol{h}^0 :=\{h^{n,0}\}_{n\in\mathbb{Z}}$. The iterates $\{u_1^{n, k}, u_2^{n, k}, h^{n, k}\}_{n=0}^N$ for $k>0$ are then extended for $n\in\mathbb{Z}$ through the iteration itself. For brevity, we retain the notation $h^{n, 0}$ and $u_1^{n, k}, u_2^{n, k}, h^{n, k}$ for the extended version. 
The exponentially weighted discrete-time Fourier transform applied to the error equations gives
\begin{equation}\label{eq:transformedDNWRIE}
\begin{aligned}
	\frac{\alpha_1}{\Delta t} (e^{r \Delta t + i \omega} \widehat u^{k+1}_1 - \widehat u^{k+1}_1) - \lambda_1  \partial_{xx} e^{r \Delta t + i \omega} \widehat u^{k+1}_1 &= 0, \\
	\widehat u^{k+1}_1(-a) = 0, \quad
	\widehat u^{k+1}_1(0) &= \widehat h^k, \\
	\frac{\alpha_2}{\Delta t} (e^{r \Delta t + i \omega} \widehat u_2^{k+1} - \widehat u_2^{k+1}) - \lambda_2  \partial_{xx} e^{r \Delta t + i \omega} \widehat u_2^{k+1} &= 0, \\
	\widehat u^{k+1}_2(b) = 0, \quad
	-\lambda_1 \partial_x \widehat u^{k+1}_1(0) &= -\lambda_2 \partial_x \widehat u^{k+1}_2(0), \\
	\widehat h^{k+1} = (1 - \theta)\,\widehat h^{k} + &\theta\,\widehat u^{k+1}_2(0).
\end{aligned}
\end{equation}
Similar to~\eqref{eq:transformedDNWR}, these are two coupled second-order ordinary differential equations in space, and solving them separately yields the closed form solutions
\[\begin{aligned}
	\widehat u_1^k(x) &= \frac{\sinh\left((x+a)\sqrt{\frac{\alpha_1(1-e^{-r \Delta t - i\omega})}{\lambda_1\Delta t}}\right)}{\sinh\left(a\sqrt{\frac{\alpha_1(1-e^{-r \Delta t -i \omega})}{\lambda_1\Delta t}}\right)} \widehat h^k, \\
	\widehat u_2^k(x) &= \sqrt{\frac{\alpha_1\lambda_1}{\alpha_2\lambda_2}}\frac{\cosh\left(a\sqrt{\frac{\alpha_1(1-e^{-r \Delta t -i \omega})}{\lambda_1\Delta t}}\right)}{\sinh\left(a\sqrt{\frac{\alpha_1(1-e^{-r \Delta t -i \omega})}{\lambda_1\Delta t}}\right)}\frac{\sinh\left((x-b)\sqrt{\frac{\alpha_2(1-e^{-r \Delta t -i \omega})}{\lambda_2\Delta t}}\right)}{\cosh\left(b\sqrt{\frac{\alpha_2(1-e^{-r \Delta t -i \omega})}{\lambda_2\Delta t}}\right)} \widehat h^k.
\end{aligned}\]
Using the last equation in~\eqref{eq:transformedDNWRIE}, we obtain by induction that
\[\widehat h^k = \hat{\rho}\left(\frac{1-e^{-r \Delta t -i \omega}}{\Delta t}\right)^k \widehat h^0,\]
where the function $\hat{\rho}$ is defined in~\eqref{eq:cvfCont}. Similar to the continuous case, the convergence factor $\hat{\rho}((1-e^{-r \Delta t -i \omega})/\Delta t)$ depends on the subdomain sizes and the physical parameters that are all hidden in the function $\hat{\rho}$. Besides, the convergence factor also depends on the time step size $\Delta t$. Our goal is to understand how the choice of time step size influences the convergence behavior, and the relationship between the continuous and time-discrete analysis. As in Theorem~\ref{2normSuperCont}, we first present an error estimate of the time-discrete iteration~\eqref{eq:DNWRIE}
for arbitrary relaxation parameter $\theta\in(0, 1)$. 

\begin{theorem}\label{superLinearTimeDisc}
For any given $\theta\in(0, 1)$, the error of the time-discrete iteration \eqref{eq:DNWRIE} satisfies the estimate 
\begin{equation}\label{eq:GenEstDis}
\|\boldsymbol{h}^k\|_{l^2} \leq \min_{r > 0} e^{rT}\max_{\omega \in [-\pi,\pi]} \left|\hat{\rho}\left(\frac{1-e^{-r \Delta t -i \omega}}{\Delta t}\right)^k\right| \|\boldsymbol{h}^0\|_{l^2},
\end{equation}
where $\hat{\rho}$ is defined in~\eqref{eq:cvfCont}.
\end{theorem}

\begin{proof}
As $r>0$, we have
\begin{align*}
	\|\boldsymbol{h}^k\|_{l^2}^2  = \sum_{n=0}^{N} e^{2 r \Delta t n}e^{-2 r \Delta t n} \left(h^{n,k}\right)^2 
	\leq e^{2 rT} \sum_{n=0}^{N}  \left(e^{-r \Delta t n}h^{n,k}\right)^2 
	\leq e^{2 rT} \sum_{n\in\mathbb{Z}} \left( e^{- r \Delta t n}h^{n,k}\right)^2,
\end{align*}
where in the last inequality we use the extension of the sequence $\boldsymbol{h}^k$. By Lemma~\ref{prepDisc2norm}, the exponentially weighted discrete-time Fourier transform of $\boldsymbol{h}^k$ exists. Applying the Parseval theorem yields 
\[\begin{aligned}
e^{2rT} \sum_{n\in\mathbb{Z}} \left( e^{- r \Delta t n}h^{n,k}\right)^2 
&= \frac{e^{2rT}}{2 \pi} \int_{-\pi}^{\pi} \left| \widehat{h}^{k}(r \Delta t + i \omega) \right|^2 d \omega \\
&= \frac{e^{2rT}}{2 \pi} \int_{-\pi}^{\pi} \left| \hat{\rho}\left(\frac{1-e^{-r \Delta t -i \omega}}{\Delta t}\right)^k \widehat{h}^{0}(r \Delta t + i \omega) \right|^2 d \omega \\
&\leq \frac{e^{2rT}}{2 \pi} \max_{\omega \in [-\pi,\pi]} \left|\hat{\rho}\left(\frac{1-e^{-r \Delta t -i \omega}}{\Delta t}\right)\right|^{2k}\int_{-\pi}^{\pi} \left| \widehat{h}^{0}(r \Delta t + i \omega) \right|^2  d \omega\\
&\leq e^{2rT} \max_{\omega \in [-\pi,\pi]} \left|\hat{\rho}\left(\frac{1-e^{-r \Delta t -i \omega}}{\Delta t}\right)\right|^{2k} \|\boldsymbol{h}^0\|_{l^2}^2,
\end{aligned}\]
where we use Lemma~\ref{prepDisc2norm} for the last inequality. Combining these computations gives
\[\|\boldsymbol{h}^k\|_{l^2}^2 \leq e^{2rT} \max_{\omega \in [-\pi,\pi]} \left|\hat{\rho}\left(\frac{1-e^{-r \Delta t -i \omega}}{\Delta t}\right)\right|^{2k} \|\boldsymbol{h}^0\|_{l^2}^2.\]
Taking the square root on both sides and minimizing over all $r > 0$ leads to~\eqref{eq:GenEstDis}.
\end{proof}

This estimate involves a min-max problem which is generally difficult to analyze. However, analogous to the continuous case, when the physical parameters in both subdomains satisfy $\sqrt{\alpha_1/\lambda_1}a = \sqrt{\alpha_2/\lambda_2}b$, the optimal relaxation parameter $\theta^*$ given in~\eqref{eq:thetaopt} again ensures convergence in two iterations for the time-discrete iteration. Furthermore, we have obtained superlinear convergence in the continuous analysis using such $\theta^*$ for more general physical conditions, {\it i.e.}, $\sqrt{\alpha_1/\lambda_1}a \neq \sqrt{\alpha_2/\lambda_2}b$. We are interested in the convergence behavior of the time-discrete iteration~\eqref{eq:DNWRIE} for this special choice of relaxation parameter $\theta^*$ and the dependence on the time step size $\Delta t$.
As in the continuous case, with the help of~\eqref{eq:rhotheta} and~\eqref{def:H}, we can rewrite the error estimate~\eqref{eq:GenEstDis} for $\theta=\theta^*$ as
\begin{equation}\label{eq:SpeEstDis}
\|\boldsymbol{h}^k\|_{l^2} \leq \left(\frac{1}{1+\sqrt{\frac{\alpha_2\lambda_2}{\alpha_1\lambda_1}}}\right)^k \min_{r > 0} e^{rT}\max_{\omega \in [-\pi,\pi]} \left|\hat H\left(\frac{1-e^{-r \Delta t -i \omega}}{\Delta t}\right)^k\right| \|\boldsymbol{h}^0\|_{l^2},
\end{equation} 
where $\alpha = \sqrt{\alpha_1/\lambda_1} a$ and $\beta = \sqrt{\alpha_2/\lambda_2} b$ in the function $\hat H$ defined in~\eqref{def:H}. Therefore, when $\theta=\theta^*$, it suffices to study the min-max problem associated with $\hat H$. We start by giving a similar result to Lemma~\ref{lem:max} for the global maximum of the maximization problem. 

\begin{lemma}\label{lem:maxDisc}
Let $\alpha> 0$, $\beta>0$, $r>0$, and $\Delta t>0$, then the function $\hat{H}$ defined in~\eqref{def:H} satisfies 
\[\max_{\omega \in [-\pi,\pi]} \left| \hat{H} \left(\frac{1-e^{-r \Delta t - i \omega}}{\Delta t}\right) \right| 
= \hat{H} \left(\frac{1-e^{-r \Delta t}}{\Delta t}\right).\]
\end{lemma}

\begin{proof}
Let $g(\omega) =(1-e^{-r \Delta t - i \omega})/ \Delta t$. For $r>0$ and $\Delta t>0$, the image of $g$ is a subset of a shifted half-plane, or more precisely
\[\text{Image}(g) = \left\{g(\omega): \, \omega \in [-\pi,\pi]\right\} \subset \left\{g(0) + x + iy \, : \,x\geq 0, y \in \mathbb{R} \right\} =: S.\]
Thus, we can transfer the maximization problem to the set $S$ as
\[\max_{\omega \in [-\pi,\pi]} \left| \hat{H}\left(g(\omega)\right) \right| \leq \max_{(x,y) \in S} \left| \hat{H} \left(g(0) + x + iy\right) \right| 
=  \max_{x \geq 0} \left| \hat H \left(g(0) + x\right) \right|,\]
where Lemma~\ref{lem:max} was used in the last equality. We now show that $\hat H$ is nonincreasing as a function of a real variable. Let $\gamma$ be a real number. Depending on the value of $\alpha$ and $\beta$, we have three cases.

\begin{enumerate}
\item[(i)] If $\alpha=\beta$, $\hat H = 0$. 

\item[(ii)] If $\alpha>\beta>0$, we have
\[\hat{H}'(\gamma) = \frac{\alpha \sinh(2 \beta \gamma) - \beta \sinh(2\alpha \gamma)}{2 \sinh^2(\alpha \gamma) \cosh^2(\beta \gamma)} < 0,\]
since $\alpha \sinh(2 \beta \gamma) < \beta \sinh(2\alpha \gamma)$. In fact, the function $f(\gamma) := \alpha \sinh(2 \beta \gamma) - \beta \sinh(2\alpha \gamma)$ satisfies
\[f(0) = 0, \quad f'(\gamma) = 2\alpha\beta(\cosh(2 \beta \gamma)-\cosh(2 \alpha \gamma)) < 0,\] 
for $\alpha > \beta >0$. 

\item[(iii)] If $\beta>\alpha>0$, we have, following a similar reasoning as in (ii),
\[\hat{H}'(\gamma) = \frac{\beta \sinh(2 \alpha \gamma) -\alpha \sinh(2\beta \gamma)}{2 \sinh^2(a \gamma) \cosh^2(b \gamma)} < 0.\]
\end{enumerate}
Therefore, we have
\[\hat{H} \left(g(0) + x\right) \leq  \hat{H} \left(g(0)\right),
\quad \forall x \geq 0.\] 
This implies that
\[\max_{\omega \in [-\pi,\pi]} \left| \hat{H}\left(g(\omega)\right) \right|	
\leq \hat{H}\left(g(0)\right) = \hat{H}\left( \frac{1-e^{-r\Delta t}}{\Delta t} \right).\] 
\end{proof}

Applying Lemma~\ref{lem:maxDisc} to~\eqref{eq:SpeEstDis}, we have 
\begin{equation}\label{eq:min-disc}
\|\boldsymbol{h}^k\|_{l^2} \leq \left(\frac{1}{1+\sqrt{\frac{\alpha_2\lambda_2}{\alpha_1\lambda_1}}}\right)^k \min_{r > 0} e^{rT} \hat H\left(\frac{1-e^{-r \Delta t}}{\Delta t}\right)^k \|\boldsymbol{h}^0\|_{l^2}.
\end{equation}
As in the continuous case, we address the minimization problem for $r$ according to the underlying physics of the problem.

\begin{theorem}[Case $\sqrt{\alpha_1/\lambda_1} a > \sqrt{\alpha_2/\lambda_2} b$]\label{thm:timeDisc-case1}

If $\theta = \theta^*$ and $\sqrt{\alpha_1/\lambda_1} a > \sqrt{\alpha_2/\lambda_2} b$, then the error of the time-discrete iterates~\eqref{eq:DNWRIE} satisfies
\begin{equation}\label{eq:DiscEst1}
\|\boldsymbol{h}^k\|_{l^2} \leq  \left(2\frac{1 - \sqrt{\frac{\alpha_2 \lambda_1}{\alpha_1\lambda_2}}\frac{b}{a}}{1 + \sqrt{\frac{\alpha_2 \lambda_2}{\alpha_1\lambda_1}}} \right)^k \phi_1(r_{1d}^*) \norm{\boldsymbol{h}^0}_{l^2},
\end{equation}
with
\begin{equation}\label{eq:roptDisc-case1}
\phi_1(r) = \exp\left( rT  - kb\sqrt{\frac{\alpha_2}{\lambda_2}}\sqrt{\frac{1-e^{-r \Delta t}}{\Delta t}}\right), \ r^*_{1d} = \frac1{\Delta t}\log\left(\frac{1 + \sqrt{\frac{\alpha_2\Delta t}{\lambda_2T^2} b^2 k^2   + 1}}{2}\right).
\end{equation}
\end{theorem}

\begin{proof}
Reusing the bound~\eqref{eq:BoundThm2} given in the proof of Theorem~\ref{thm:cont-case1} with $r = (1-\exp(-r \Delta t))/\Delta t$ yields
\begin{equation}\label{eq:minBoundThm6}
\left(\frac1{1+\sqrt{\frac{\alpha_2\lambda_2}{\alpha_1\lambda_1}}}\right)^k\min_{r>0}e^{rT} \hat H \left(\frac{1-e^{-r \Delta t}}{\Delta t}\right)^k 
\leq \left(2\frac{\bigl|1 - \sqrt{\frac{\alpha_2 \lambda_1}{\alpha_1\lambda_2}}\frac{b}{a}\bigr|}{1 + \sqrt{\frac{\alpha_2 \lambda_2}{\alpha_1 \lambda_1}}} \right)^k\min_{r > 0} \phi_1(r).
\end{equation}
 Differentiating $\phi_1$ yields
\[\phi_1'(r) = \left(T - \frac{kb}{2}\sqrt{\frac{\alpha_2}{\lambda_2}}\sqrt{\frac{\Delta t}{1 - e^{-r \Delta t}}} e^{-r \Delta t} \right)\exp\left(r T -  kb\sqrt{\frac{\alpha_2}{\lambda_2}} \sqrt{\frac{1 - e^{-r \Delta t}}{\Delta t}} \right).\]
Setting the latter to 0 yields the extremal point $r^*_{1d}$ given in~\eqref{eq:roptDisc-case1}. The derivative $\phi_1'(r)$ is negative for $r<r^*_{1d}$ and positive for $r>r^*_{1d}$, since $e^{-r \Delta t}\sqrt{\Delta t/(1 - e^{-r \Delta t})}$ is a decreasing function of $r$. This implies that the extremum is indeed a minimum. Inserting $r^*_{1d}$ into $\phi_1$ in~\eqref{eq:minBoundThm6} gives the estimate~\eqref{eq:DiscEst1}.
\end{proof}

\begin{theorem}[Case $\sqrt{\alpha_1/\lambda_1} a < \sqrt{\alpha_2/\lambda_2} b$]\label{thm:timeDisc-case2}

Let $\theta = \theta^*$. If $\sqrt{\alpha_1/\lambda_1} a < \sqrt{\alpha_2/\lambda_2} b < 2\sqrt{\alpha_1/\lambda_1} a$, then the error of the iterates~\eqref{eq:DNWRIE} satisfies
\begin{equation}\label{eq:DiscEst2-1}
\|\boldsymbol{h}^k\|_{l^2} \leq \left(2\frac{\sqrt{\frac{\alpha_2 \lambda_1}{\alpha_1\lambda_2}}\frac{b}{a} -1}{\sqrt{\frac{\alpha_2 \lambda_2}{\alpha_1\lambda_1}} + 1} \right)^k \phi_1(r^*_{1d}) \|\boldsymbol{h}^0\|_{l^2},
\end{equation}
with $r^*_{1d}$ and $\phi_1$ defined in~\eqref{eq:roptDisc-case1}. 

If $\sqrt{\alpha_2/\lambda_2} b \geq 2\sqrt{\alpha_1/\lambda_1} a$, then the error of the iterates~\eqref{eq:DNWRIE} satisfies
\begin{equation}\label{eq:DiscEst2-2}
\|\boldsymbol{h}^k\|_{l^2} \leq \left(2\frac{\sqrt{\frac{\alpha_2 \lambda_1}{\alpha_1\lambda_2}}\frac{b}{a} - 1}{\sqrt{\frac{\alpha_2 \lambda_2}{\alpha_1\lambda_1}}+1} \right)^k 
\phi_2(r_{2d}^*) \|\boldsymbol{h}^0\|_{l^2},
\end{equation}
with 
\begin{equation}\label{eq:roptDisc-case2}
\phi_2(r) = \exp\left( r T  - ka\sqrt{\frac{\alpha_1}{\lambda_1}}\sqrt{\frac{1-e^{-r \Delta t}}{\Delta t}}\right), \  r^*_{2d} = \frac1{\Delta t}\log\left(\frac{1 + \sqrt{4\frac{\alpha_1  \Delta t}{ \lambda_1 T^2}a^2 k^2 + 1}}{2}\right).
\end{equation}
\end{theorem}

\begin{proof}
The proof for the case $\sqrt{\alpha_1/\lambda_1} a < \sqrt{\alpha_2/\lambda_2} b < 2\sqrt{\alpha_1/\lambda_1} a$ follows the same steps as that of Theorem~\ref{thm:timeDisc-case1}, except we make use of the bound~\eqref{eq:BoundThm3}. This leads to the same optimized $r^*_{1d}$ given in~\eqref{eq:roptDisc-case1}.

For $\sqrt{\alpha_2/\lambda_2} b \geq 2\sqrt{\alpha_1/\lambda_1} a$, we follow once again the proof of Theorem~\ref{thm:timeDisc-case1}, but use instead  the bound~\eqref{eq:bound-exp}, which gives
\[\left(\frac1{1+\sqrt{\frac{\alpha_2\lambda_2}{\alpha_1\lambda_1}}}\right)^k\min_{r>0}e^{rT} \hat H \left(\frac{1-e^{-r \Delta t}}{\Delta t}\right)^k 
\leq \left(2\frac{\frac{\sqrt{\alpha_2 \lambda_1}b}{\sqrt{\alpha_1\lambda_2}a} - 1}{\sqrt{\frac{\alpha_2 \lambda_2}{\alpha_1\lambda_1}}+1} \right)^k \min_{r > 0} \phi_2(r).\]
Setting then $\phi'_2(r)=0$, we obtain the optimum value of $r^*_{2d}$~\eqref{eq:roptDisc-case2}. Inserting it into $\phi_2$ gives the result~\eqref{eq:DiscEst2-2}.
\end{proof}

Theorems~\ref{thm:timeDisc-case1} and~\ref{thm:timeDisc-case2} provide time-discrete error estimates that depend on both time $T$ and step size $\Delta t$. Note that $\phi_1(r^*_{1d})$ and $\phi_2(r^*_{2d})$ are guaranteed to be less than or equal to 1, since $\phi_1$ and $\phi_2$ are continuous for $r>0$ and $1 =\phi_1(0) \geq \min_{r>0} \phi_1(r) = \phi_1(r^*_{1d})$ and $1 =\phi_2(0) \geq \min_{r>0} \phi_2(r) = \phi_2(r^*_{2d})$. Thus, the exponential factor is guaranteed to be nonincreasing. 

Furthermore, the time-discrete estimates share the same linear factor as in the continuous estimates given in Theorems~\ref{thm:cont-case1} and~\ref{thm:cont-case2}. Similarly to the continuous case, the weights~\eqref{eq:roptDisc-case1} and~\eqref{eq:roptDisc-case2} depend on the subdomain size and physical parameters. As they are auxiliary parameters that are only used in our analysis to derive error estimates, we can also substitute them directly into~\eqref{eq:min-cont} to obtain sharper error estimates, as in Corollary~\ref{thm:min-cont-bis}.

\begin{corollary}\label{thm:min-disc-bis}
	For $\theta = \theta^*$, the error of the iterates~\eqref{eq:DNWRIE} satisfies
	\begin{equation}\label{eq:min-disc-bis}
			\|\boldsymbol h^k\|_{l^2} 
			\leq \left(\frac{\hat H \left(\frac{1-e^{-r^* \Delta t}}{\Delta t}\right)}{1+\sqrt{\frac{\alpha_2\lambda_2}{\alpha_1\lambda_1}}}\right)^ke^{r^*T} \|\boldsymbol{h}^0\|_{l^2},
	\end{equation} 
	where $r^*$ denotes $r_{1d^*}$ or $r_{2d^*}$ depending on the value of  $\sqrt{\alpha_1/\lambda_1} a$ and $\sqrt{\alpha_2/\lambda_2} b$.
\end{corollary}

Once again, the error estimate~\eqref{eq:min-disc-bis} is sharper than those in Theorems~\ref{thm:timeDisc-case1} and~\ref{thm:timeDisc-case2}, since we directly inject the value $r^*$ (i.e., $r_{1d}^*$ or $r_{2d}^*$) into the function $\hat H$ instead of bounding it first. However, its dependence on the time step size $\Delta t$ is hidden in the value of $r^*$. To compare with the estimates obtained in the continuous case and to understand the impact of $\Delta t$, we still need to use the estimates derived in Theorems~\ref{thm:timeDisc-case1} and~\ref{thm:timeDisc-case2}.

%% file: compare.tex
\section{Comparison of continuous and time-discrete estimates}\label{sec:4}

We now discuss the relation of the continuous and time-discrete estimates. 

\begin{corollary}\label{cor:discLarger}
  For $\theta = \theta^*$ defined in~\eqref{eq:thetaopt}, the time-discrete error estimate~\eqref{eq:min-disc} is larger than~\eqref{eq:min-cont}.
\end{corollary}
\begin{proof}
For $\theta = \theta^*$, the convergence factor can be simplified using $\hat{H}$ defined in~\eqref{def:H}. By Lemmas~\ref{lem:max} and~\ref{lem:maxDisc}, the maximum of $\hat H$ in the two cases are $\hat H(r)$ and $\hat H((1-e^{-r\Delta t})/\Delta t)$. As shown in the proof of Lemma~\ref{lem:maxDisc}, $\hat{H}$ is nonincreasing for real values, and we always have $(1-e^{-r \Delta t})/\Delta t \leq r$. This implies that \[\hat H\left(\frac{1-e^{-r\Delta t}}{\Delta t}\right) \geq \hat H(r),\] 
and thus we obtain the result. 
\end{proof}

As shown in Theorems~\ref{thm:cont-case1}, \ref{thm:cont-case2}, \ref{thm:timeDisc-case1} and \ref{thm:timeDisc-case2}, the optimal $r^*$ is obtained based on the choice of the material parameters and domain sizes. In the case $\sqrt{\alpha_1/\lambda_1} a  >  \sqrt{\alpha_2/\lambda_2} b$, the optimal $r^*$ are $r^*_{1c}$~\eqref{eq:ropt-case1} and $r^*_{1d}$~\eqref{eq:roptDisc-case1}. Taylor expanding $r^*_{1d}$ in $\Delta t\alpha_2b^2 k^2/(\lambda_2T^2)$ gives 
\begin{equation}\label{eq:taylor}
r^*_{1d} = \frac1{\Delta t}\log\left(\frac{1 + \sqrt{\frac{\alpha_2\Delta t}{\lambda_2T^2} b^2 k^2   + 1}}{2}\right) = \underbrace{\frac{\alpha_2 b^2 k^2}{4 \lambda_2 T^2}}_{=r^*_{1c}} + O\left(\Delta t \left(\frac{\alpha_2 b^2 k^2}{\lambda_2 T^2}\right)^2\right).
\end{equation}
Thus, $r^*_{1d}$ is close to $r^*_{1c}$ if $\Delta t \left(\alpha_2 b^2 k^2/(\lambda_2 T^2)\right)^2$ is small (similar results are obtained by expanding $r^*_{2d}$ in $\Delta t \alpha_1 a^2 k^2/(\lambda_1T^2)$ and comparing with $r^*_{2c}$). 
This acts as a scaling determining which $\Delta t$ respectively $k^2/T$ are considered small or large.  
%
%

Now we compare the time-discrete estimates obtained in Theorems~\ref{thm:timeDisc-case1} and~\ref{thm:timeDisc-case2} with their continuous counterparts Theorems~\ref{thm:cont-case1} and~\ref{thm:cont-case2}. 

\begin{corollary}\label{cor:dtTo0}
For given positive parameters $\alpha_j, \lambda_j$, $j=1,2$, $a, b, T$, and $k>0$, when $\Delta t\to 0$, the time-discrete estimates given in Theorems~\ref{thm:timeDisc-case1} and~\ref{thm:timeDisc-case2} converge to the continuous estimates in Theorems~\ref{thm:cont-case1} and~\ref{thm:cont-case2}, respectively.
\end{corollary}

\begin{proof}
We consider the estimates in Theorems~\ref{thm:timeDisc-case1} and~\ref{thm:cont-case1} to show the convergence.
An expansion of $\phi_1(r^*_{1d})$ defined in~\eqref{eq:roptDisc-case1} for $\Delta t\to 0$ gives
\begin{equation}\label{eq:discsmalldt}
\phi_1(r^*_{1d}) = e^{-\frac{\alpha_2 b^2k^2}{4\lambda_2 T}}\left(1 + \frac{\alpha_2^2b^4k^4}{32\lambda_2^2T^3}\Delta t\right) + O\left(\Delta t^2\right).
\end{equation}
Therefore, we obtain that
\[\lim_{\Delta t\to0} \phi_1(r^*_{1d}) = e^{-\frac{\alpha_2 b^2k^2}{4\lambda_2 T}},\]
and the limit is the exponential part in~\eqref{eq:ContEst1}. 
It follows that the time-discrete estimate~\eqref{eq:DiscEst1} converges to the continuous one~\eqref{eq:ContEst1} when $\Delta t\to 0$.
\end{proof}

In~\eqref{eq:discsmalldt}, we can rewrite the second term in the expansion as 
\[\frac{\alpha_2^2b^4k^4}{32\lambda_2^2T^3}\Delta t = \frac{\alpha_2^2b^4}{32\lambda_2^2}\left(\frac{k^2}{T}\right)^2\frac{\Delta t}T.\]
When $\Delta t\to 0$, $(k^2/T)^2(\Delta t/T)$ becomes small. 
Using Corollary~\ref{cor:cont}, we then find
\begin{itemize}
\item when $k^2/T$ is small, the time-discrete estimates show linear convergence;  
\item when $k^2/T$ is large, the time-discrete estimates show superlinear convergence.
\end{itemize}

The expansion~\eqref{eq:discsmalldt} for $\Delta t\to 0$ holds true for a fixed iteration count $k$, when $k\to \infty$, we need the following result.

\begin{corollary}\label{cor:klarge}
For given positive parameters $\alpha_j, \lambda_j$, $j=1,2$, $a, b, T$, and $\Delta t>0$, when $k\to\infty$, the time-discrete estimates in Theorems~\ref{thm:timeDisc-case1} and~\ref{thm:timeDisc-case2} behave linearly.
\end{corollary}

\begin{proof}
We consider the estimate in Theorem~\ref{thm:timeDisc-case1}, the other follows analogously. 
For large $k$, the asymptotic behavior of $\phi_1(r^*_{1d})$ is given by
\begin{equation}\label{eq:klarge}
\phi_1(r^*_{1d}) = \left(\frac b{2T}\sqrt{\frac{\alpha_2\Delta t}{\lambda_2}} \right)^{\frac{T}{\Delta t}}e^\frac{T}{\Delta t}k^{\frac{T}{\Delta t}}\left(e^{-b\sqrt{\frac{\alpha_2}{\lambda_2 \Delta t}}}\right)^k\left[1+\sqrt{\frac{\lambda_2}{\alpha_2}}\frac{T^2}{2\Delta t^{\frac32}bk} + O\left(\frac1{k^2}\right)\right].
\end{equation}
In particular, for $c>0$ and $0<\gamma<1$, $k^c\gamma^k$ behaves asymptotically as $\gamma^k$ for large $k$.
Therefore, when $k\to\infty$, the time-discrete estimate given in~\eqref{eq:DiscEst1} behaves linearly as
\[ \left(2\frac{1 - \sqrt{\frac{\alpha_2 \lambda_1}{\alpha_1\lambda_2}}\frac{b}{a}}{1 + \sqrt{\frac{\alpha_2 \lambda_2}{\alpha_1\lambda_1}}} \right)^k \phi_1(r_{1d}^*) \sim \left(2\frac{1 - \sqrt{\frac{\alpha_2 \lambda_1}{\alpha_1\lambda_2}}\frac{b}{a}}{1 + \sqrt{\frac{\alpha_2 \lambda_2}{\alpha_1\lambda_1}}} \right)^k \left(e^{-b\sqrt{\frac{\alpha_2}{\lambda_2 \Delta t}}}\right)^k.\]
\end{proof}

These theoretical results are illustrated by the following two numerical test cases:
\begin{itemize}
	\item {\bf Case A:} a low contrast test case with physical material parameters $\alpha_1=4$, $\alpha_2=\lambda_1=\lambda_2=1$. The physical parameters satisfy the condition $\sqrt{\alpha_1/\lambda_1} a = 2 > 1= \sqrt{\alpha_2/\lambda_2} b$.
	
	\item {\bf Case B:} a more realistic high contrast case with physical material parameters for air $\alpha = 1299$ and $\lambda = 0.0243$, and for steel $\alpha = 3.47 \times 10^6$ and $\lambda = 48.9$. The air and steel domains can be arranged either as air-steel or steel-air, depending on whether the Dirichlet condition is imposed on the air or the steel domain.
\end{itemize}
For all our tests, we use the relaxation parameter $\theta^*$ defined in~\eqref{eq:thetaopt} and consider $a = b = 1$, {\it i.e.}, the space domain $(-1, 1)$. 

To get better insights into both continuous and time-discrete estimates, we illustrate their behavior in Figure~\ref{fig:klarge} with $\Delta t=20$ and $T = 1000$.
\begin{figure}
\begin{center}
\includegraphics[scale = 0.3]{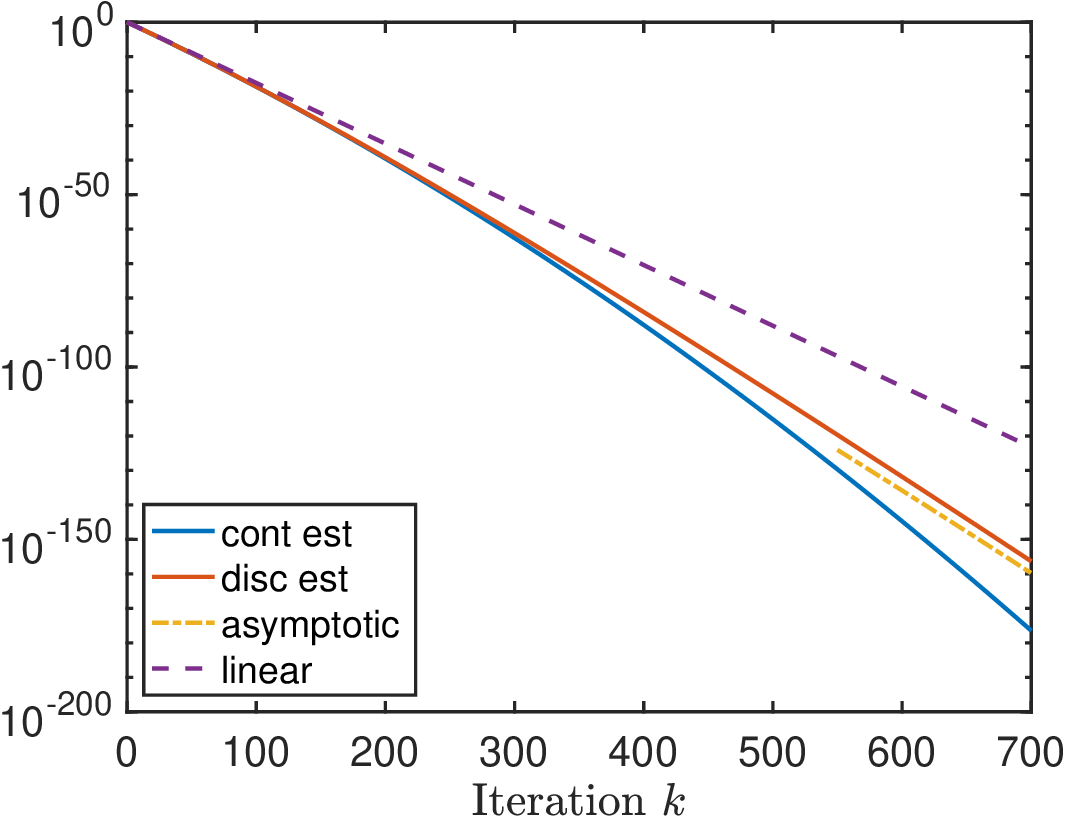}
\includegraphics[scale = 0.3]{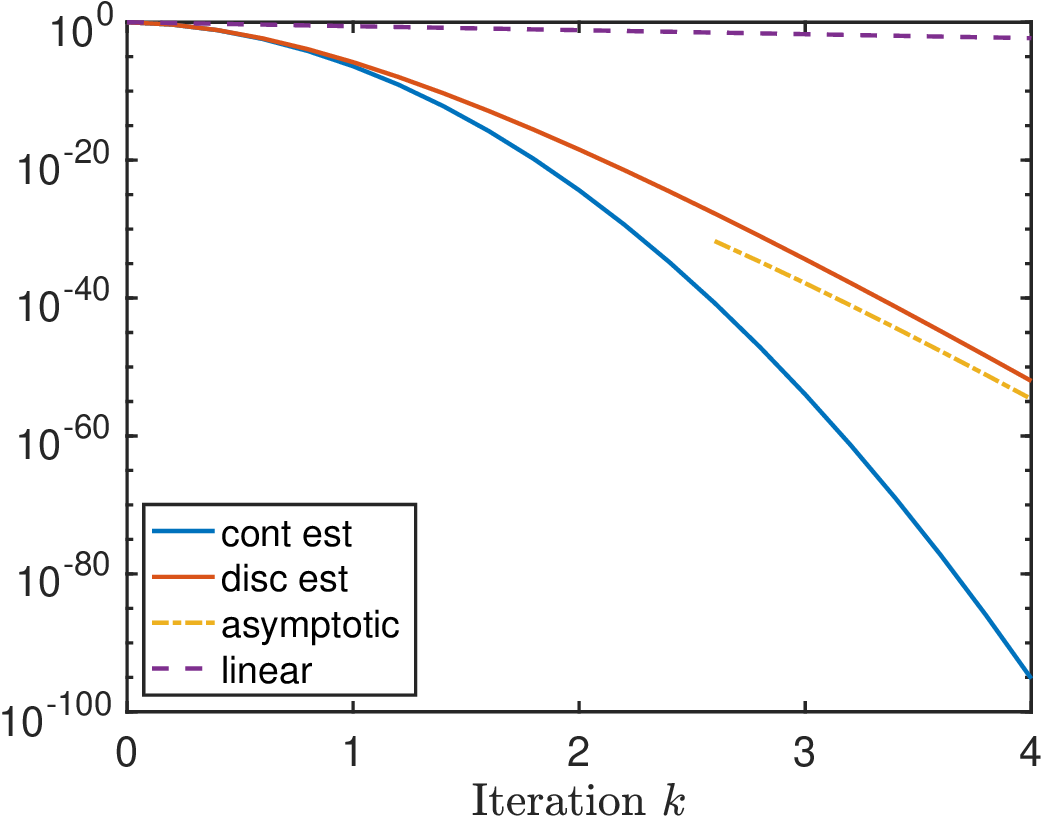}
\end{center}
\caption{Illustration of the continuous and time-discrete estimates for {\bf Case A} (left) and {\bf Case B} (right) with $\Delta t=20$ and $T = 1000$.
The blue curve represents the continuous estimate defined in~\eqref{eq:ContEst1}, the red curve is the time-discrete estimate defined in~\eqref{eq:DiscEst1}, the yellow curve is the asymptotic behavior of the time-discrete estimate in Corollary~\ref{cor:klarge}, and the purple curve is the linear factor in both continuous and time-discrete estimates.}
\label{fig:klarge}
\end{figure}
We observe that the time discrete estimates are larger than the continuous estimates as is explained by Corollary \ref{cor:discLarger}. In {\bf Case A}, the continuous estimate~\eqref{eq:ContEst1} behaves linearly up to iteration $k =100$, and then becomes superlinear. 
The time-discrete estimate~\eqref{eq:DiscEst1} follows the continuous estimate until iteration $k = 300$, and then enters the linear regime as shown in Corollary~\ref{cor:klarge}.
In this regime where $k$ is large, we see that the red line approaches the yellow line from above, since the correction term in~\eqref{eq:klarge} is positive.
In {\bf Case B}, both estimates predict extremely fast convergence, because of the large contrast between the material parameters. The continuous estimate behaves superlinearly and the time-discrete folls for only one iteration before entering the linear regime. 

We next solve the minimization problems given in~\eqref{eq:min-cont} and~\eqref{eq:min-disc} numerically, using the Nelder-Mead optimization algorithm in {\tt SciPy} with an absolute tolerance of $10^{-6}$. The code for these tests can be found in the GitHub repository~\cite{code}.

In {\bf Case A}, we compare the estimates in \eqref{eq:min-cont}, \eqref{eq:ContEst1} and~\eqref{eq:min-cont-bis} for the continuous case, and~\eqref{eq:min-disc}, \eqref{eq:DiscEst1} and~\eqref{eq:min-disc-bis} for the time-discrete case.  We set $T=1/4$ and $\Delta t = 1/64$ and plot them as a function of the iteration count $k$ in Figure~\ref{fg:differentErrorEsts}. 
\begin{figure}[t]
	\centering
	\includegraphics[width=0.5\textwidth]{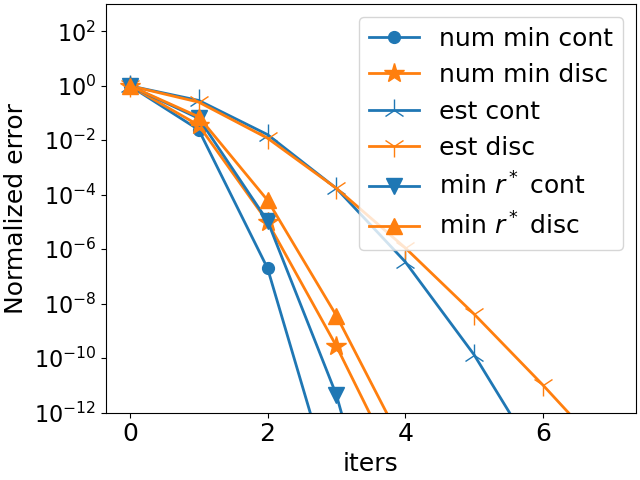}
	\caption{The continuous and time-discrete explicit error estimates as well as the improved version for {\bf Case A} with $\Delta t=1/64$ and $T=1/4$. As a baseline this plot also includes numerically calculated error estimates from Theorem \ref{2normSuperCont} and \ref{superLinearTimeDisc}. }
	\label{fg:differentErrorEsts}
\end{figure}
For both continuous and time-discrete error bounds~\eqref{eq:ContEst1} and~\eqref{eq:DiscEst1}, we observe superlinear behavior as expected, since the time $T$ is small. However, they are both less accurate compared to~\eqref{eq:min-cont} and~\eqref{eq:min-disc}. To improve it, we use the error bounds in~\eqref{eq:min-cont-bis} and~\eqref{eq:min-disc-bis}. The latter two are much sharper in both cases compared with~\eqref{eq:ContEst1} and~\eqref{eq:DiscEst1}. Thus, both~\eqref{eq:min-cont-bis} and~\eqref{eq:min-disc-bis} can provide more accurate predictions of the convergence behavior with the help of the auxiliary parameter $r^*$. In the following experiments, we will use~\eqref{eq:min-cont-bis} and~\eqref{eq:min-disc-bis}. 

%% file: experiments.tex
\section{Numerical experiments}\label{sec:5}

We now compare the continuous~\eqref{eq:min-cont-bis} and time-discrete estimates~\eqref{eq:min-disc-bis} with the observed convergence behavior of DNWR iterations. We first study a one dimensional case corresponding to the problem described in~\eqref{eq:DNWR} and then extend our study to a two dimensional nonlinear case.

\subsection{One dimensional experiment}
We use the implicit Euler method in time and the centered finite difference method in space. We consider a very fine spatial mesh size $\Delta x = 10^{-4}$ to ensure that the effect of the time discretization dominates over the spatial discretization. We set the relaxation parameter $\theta = \theta^*$ defined in~\eqref{eq:thetaopt}. The initial waveform error $\boldsymbol h^0=\{h^{n,0}\}_{n=0}^N$ is set to be random, as discussed in~\cite[Section 5.1]{gander}. The Python code for these experiments is based on~\cite{Me25} and can be found in the GitHub repository~\cite{code}. Figure~\ref{fg:simpleSol} shows the error of the first and the third iteration for {\bf Case A} with $T=2$ and $\Delta t = 1/64$. 
\begin{figure}[t]
	\centering
	\mbox{\includegraphics[width=0.49\textwidth]{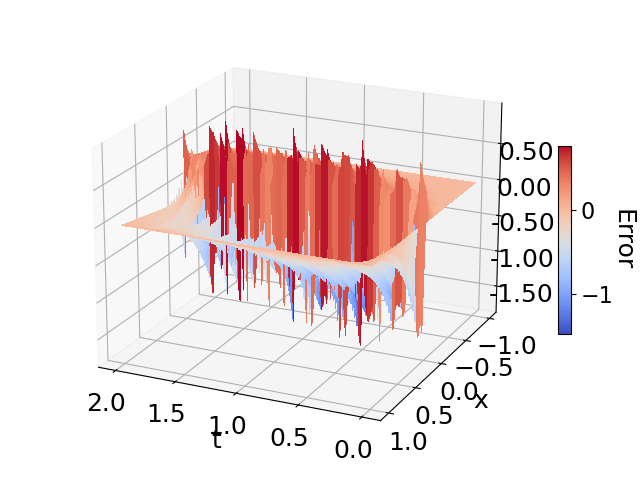}
		\includegraphics[width=0.49\textwidth]{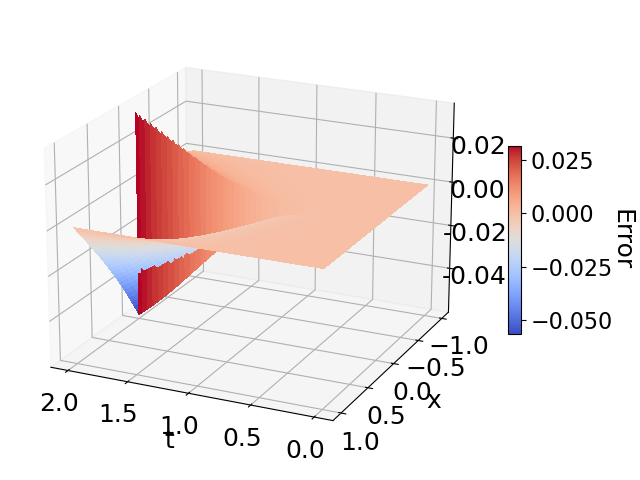}}
	\caption{Error of the DNWR iterations for {\bf Case A} with $T=2$ and $\Delta t = 1/64$. Left: First iteration. Right: Third iteration.}
	\label{fg:simpleSol}
\end{figure}
On the left of Figure~\ref{fg:simpleSol}, the waveform error $\boldsymbol h^0$ after the first iteration is mainly centered around the interface $x=0$ with nonsmooth information. The error is decreased by a factor of $100$ after the third iteration and becomes very smooth. In addition, it increases in $t$ and attains its maximum at $T=2$. This is expected, as the error estimates in both continuous and time-discrete cases are increasing with respect to $T$.

We then compare the observed convergence behavior with both continuous and time-discrete estimates in Figure~\ref{fg:simpleFullyDisc} for {\bf Case A}. 
\begin{figure}[t]
	\centering
	\mbox{\includegraphics[width=0.32\textwidth]{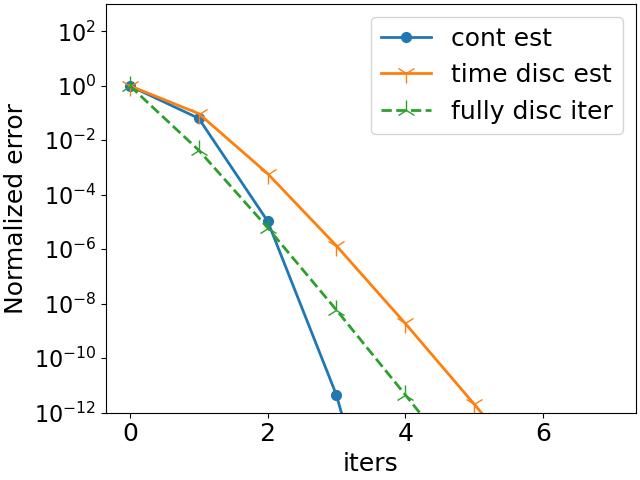}
		\includegraphics[width=0.32\textwidth]{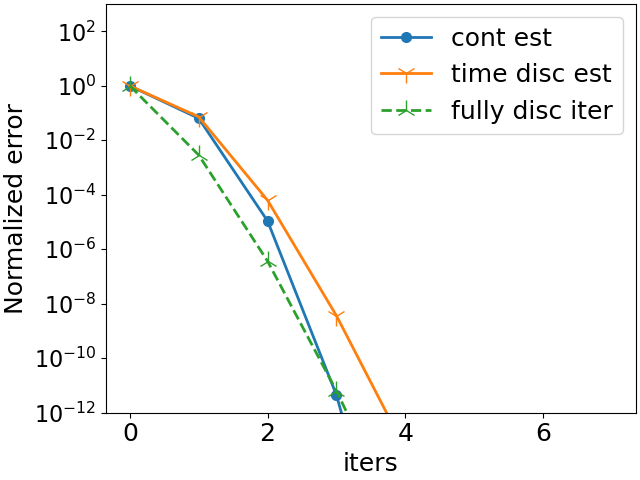}
	\includegraphics[width=0.32\textwidth]{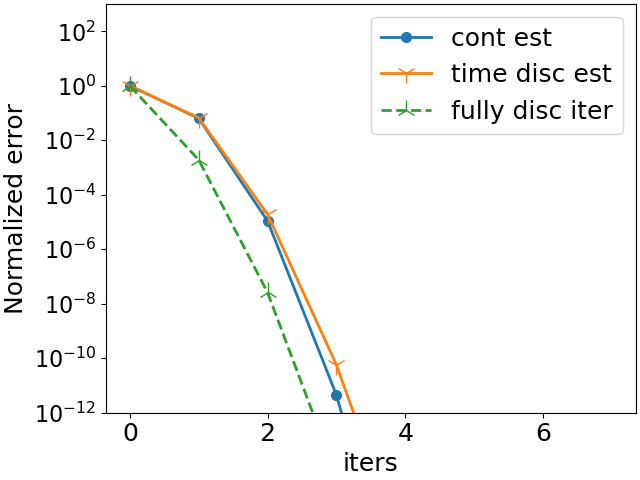}}
	\mbox{\includegraphics[width=0.32\textwidth]{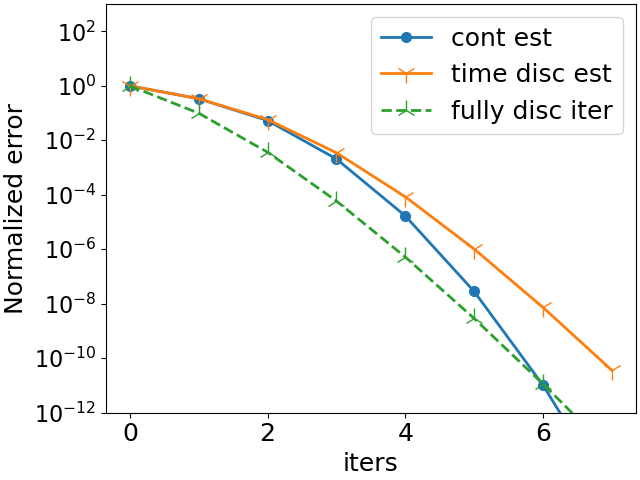}
	\includegraphics[width=0.32\textwidth]{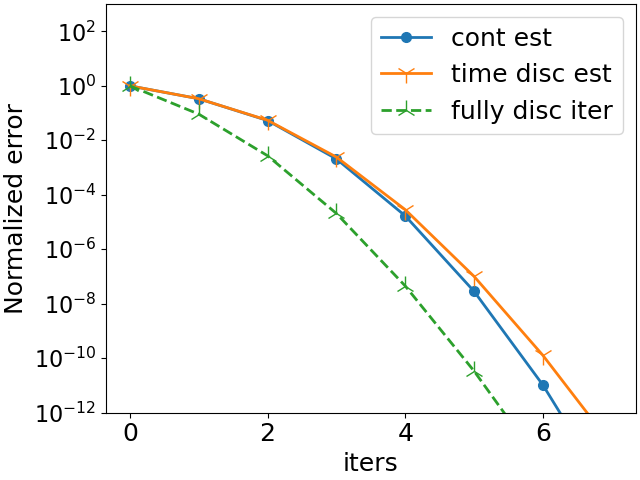}
	\includegraphics[width=0.32\textwidth]{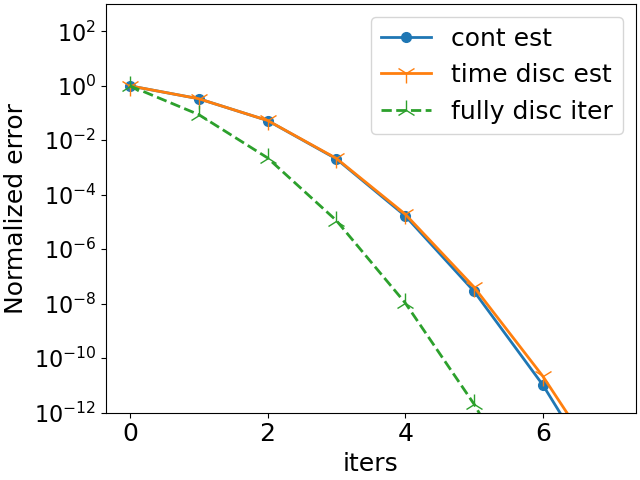}
	\hspace{12em}}
	\caption{Comparison of continuous error estimate (blue), time-discrete error estimates (orange) and normalized error $\|\boldsymbol h^k|_{l^2}/\|\boldsymbol h^0|_{l^2}$ of DNWR iterations (green) for {\bf Case A}. Top: $T = 1/4$. Bottom $T=1$. 
	Left: $\Delta t= 1/16$. Middle: $\Delta t = 1/64$. Right: $\Delta t= 1/256$. 
	}
	\label{fg:simpleFullyDisc}
\end{figure}
In this test, we consider two times $T=1/4$ and $T=1$, and three time steps $\Delta t = 1/16, 1/64, 1/256$. We observe that the time-discrete estimate accurately captures the convergence behavior of the DNWR iterations for all cases, and the method converges more slowly than predicted by the continuous estimate, when $\Delta t$ is large, and the numerical solutions are not accurate.
For small $\Delta t$, when the resulting numerical solutions become more accurate and close to the true solutions, convergence becomes as fast as predicted by the continuous estimate. We also see that the time-discrete error estimate approaches the continuous one when $\Delta t \to 0$, which is explained by Corollary \ref{cor:dtTo0}. Comparing the top and bottom row in Figure~\ref{fg:simpleFullyDisc}, we also see that the continuous estimate becomes more accurate when $\Delta t$ is kept fixed and $T$ increases. This is consistent with the series expansion of $r^*_{1d}$ in the time discrete case in \eqref{eq:taylor}, which yields that $r^*_{1d}$ is close to $r^*_{1c}$ when $T$ is large.

We next use the high contrast steel-air {\bf Case B} to illustrate that one has to assign the Dirichlet and Neumann conditions accordingly to get fast convergence. Figure \ref{fg:realisticFullyDisc} shows the result for $T=5 \times 10^{4}$, at the top the air-steel case, and at the bottom the steel-air case. To illustrate the behaviour for small and large time steps, we use $\Delta t =10, 10^4$.
 

\begin{figure}[t]
	\centering
	\mbox{\includegraphics[width=0.33\textwidth]{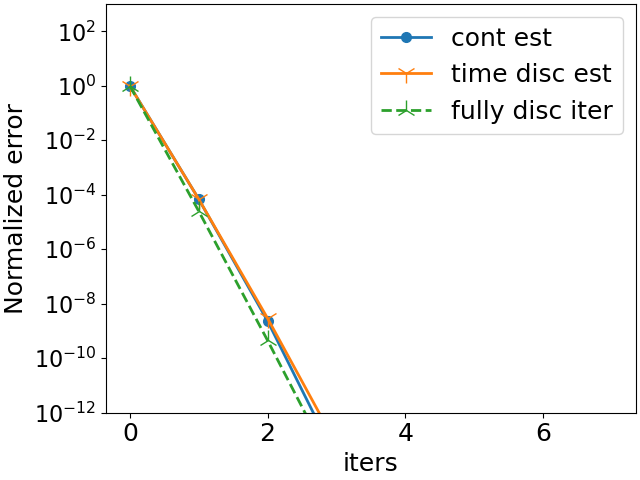}
	\includegraphics[width=0.33\textwidth]{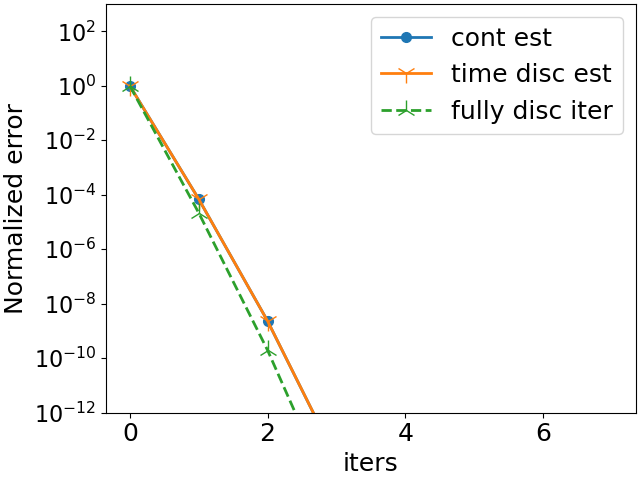}}
	\mbox{\includegraphics[width=0.33\textwidth]{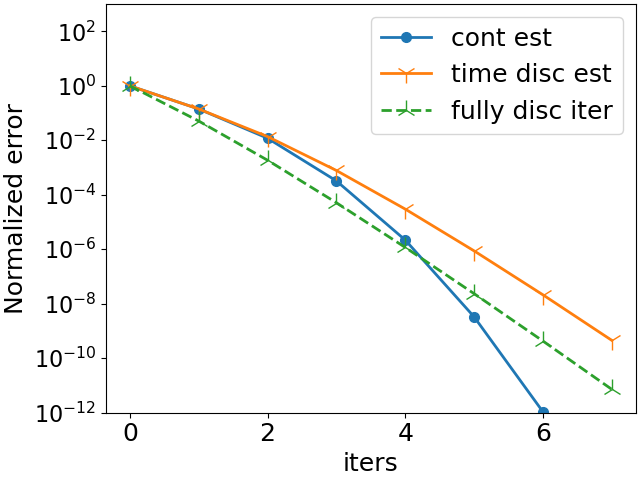}
	\includegraphics[width=0.33\textwidth]{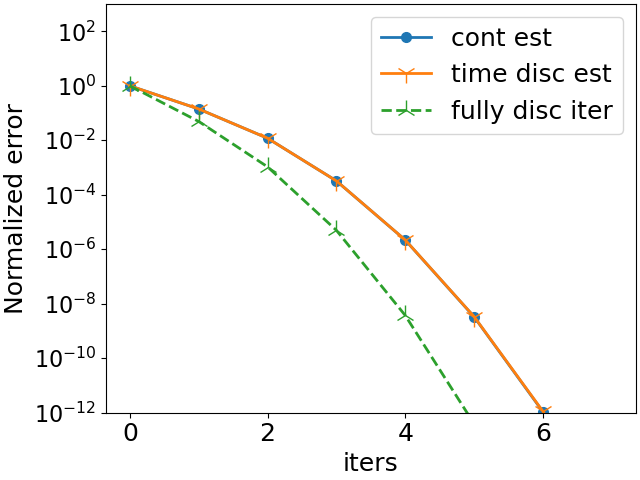}}
	\caption{Comparison of continuous error estimate (blue), time-discrete error estimates (orange) and normalized error $\|\boldsymbol h^k\|_{l^2}/\|\boldsymbol h^0\|_{l^2}$ of~\eqref{eq:DNWR} iterations (green) for {\bf Case B} with $T=5 \times 10^4$. 
	Left: $\Delta t= 10^4$. Right: $\Delta t=10$. Top: air-steel. Bottom: steel-air. 
	}
	\label{fg:realisticFullyDisc}
\end{figure}

We see that convergence for the air-steel case is much faster than the steel-air case, and both the continuous and time-discrete estimates predict this well. In the much slower steel-air case, the behavior is similar to Figure~\ref{fg:simpleFullyDisc}.

\subsection{Flow over hot plate}

To investigate the relevance of our continuous and time-discrete error estimates for more complex cases, we consider a nonlinear thermal Fluid-Structure interaction (FSI) test case, consisting of a hot steel plate cooled by a cold air stream, see Figure \ref{fg:diag}. This setup is widely used as a benchmark, see e.g. \cite{chourdakis2022, vynnycky1998, monbir:18}. 

\begin{figure}
			\includegraphics[width=0.9\textwidth, trim= 0 60 0 2,clip]{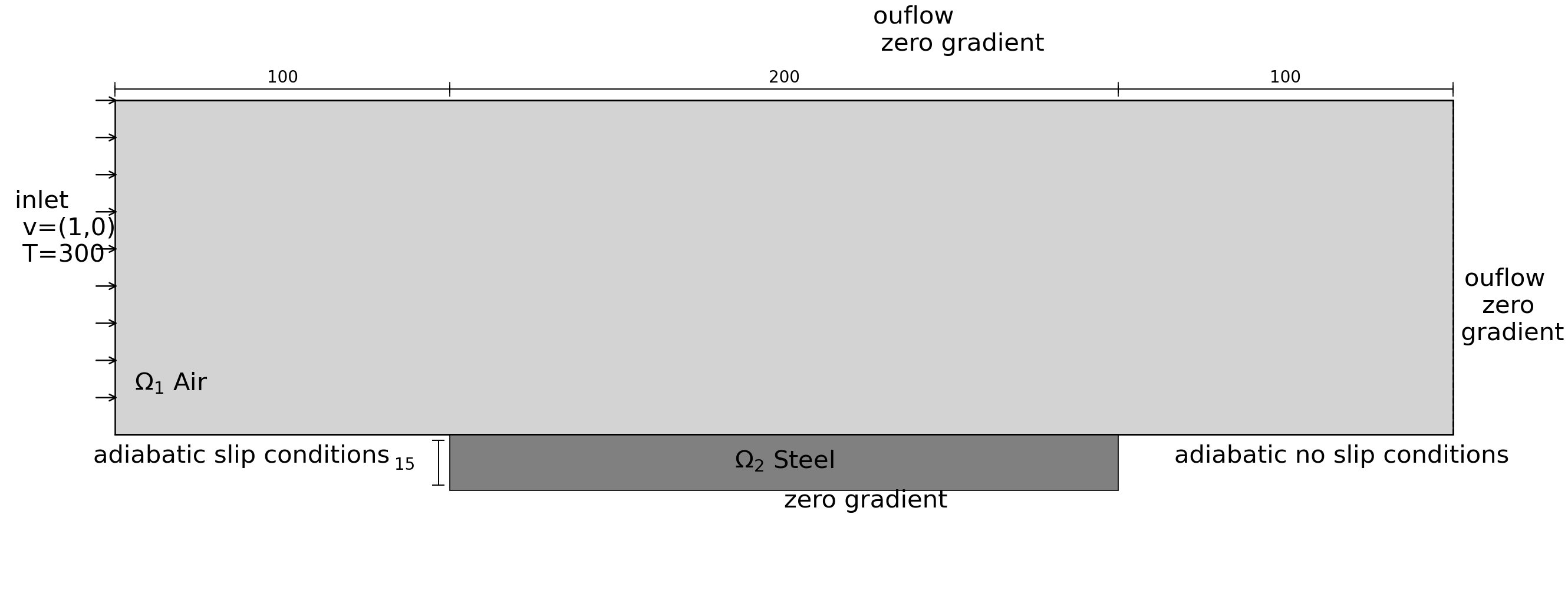}
	\caption{Diagram of the FSI test case, consisting of a hot steel plate cooled by a cold air stream.	\label{fg:diag}}
\end{figure}

The air is modeled using the time dependent Reynolds Averaged Navier-Stokes equations with the $k$-$\varepsilon$ turbulence model in the two dimensional domain $\Omega_1 = (-0.1,0.3)\times(0,0.1)$, see e.g. \cite{blazek2005, birken:21}. The temperature $u$ is related to the dynamic viscosity through Sutherland’s law,
\[\mu(u) = \hat{\mu}_{ref} (u/\hat{u}_{ref})^{3/2} \frac{\hat{u}_{ref} + \hat{S}}{u + \hat{S}}, \] 
where $\hat{S} = 110$, $\hat{\mu}_{ref} = 1.8 \times 10^{-5}$ and $\hat{u}_{ref} =273$. The fluid boundary conditions are shown in Figure \ref{fg:diag}. The initial condition is a steady-state solution computed with a fixed interface temperature of $900 K$.

\begin{figure}
		\includegraphics[width=1.0\textwidth,trim= 0 30 0 20,clip]{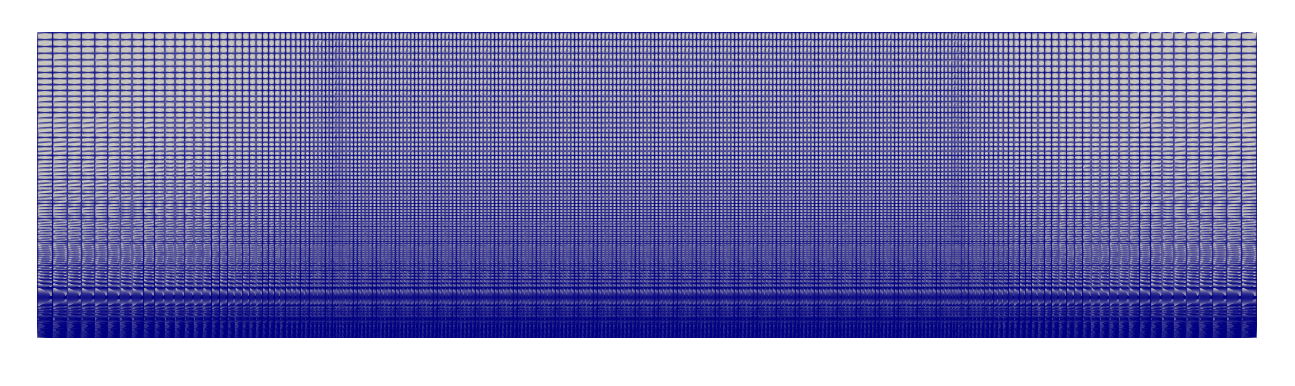}
		\includegraphics[width=1.0\textwidth,trim= 0 35 0 20,clip]{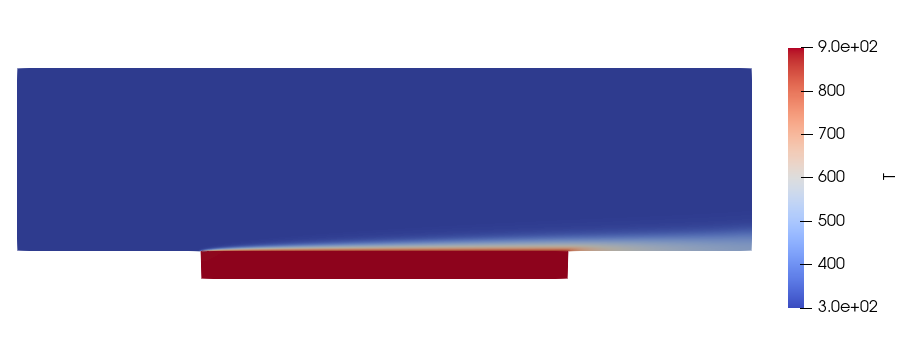}
	\caption{Top: The air and steel mesh. Bottom: The solution at $T=100$ for $\Delta t = 1$. 	\label{meshes}}
\end{figure}

As a solver for the fluid we use OpenFOAM \cite{weller1998} with a cell-centered finite-volume method in space, implicit Euler in time and the PIMPLE algorithm as nonlinear solver. The mesh is generated using the blockMesh utility with 240 cells in the horizontal direction and 200 cells in the vertical direction. To adequately resolve the boundary layer, a non-uniform mesh is used, obtained by the simpleGrading function in blockMesh, with a ratio of 100 between the bottom and top cells. This corresponds to interface grid sizes of $2.3 \times 10^{-5}$. The mesh is shown in Figure \ref{meshes}.

The steel plate is modeled using the nonlinear heat equation
\begin{equation*}
	\rho c_p(u) u({\bf x},t) -\nabla \cdot (\lambda(u) \nabla u({\bf x},t)) = 0, \qquad \text{ on } \Omega_2,
\end{equation*}
where  $ \Omega_2 = (0,0.2)\times(-0.015,0)$, $\rho$ denotes the density, $c_p(u)$ the temperature dependent heat capacity and $\lambda(u)$ the temperature dependent heat conductivity. The following empirical model for steel 51CrV4 from \cite{quint2011} is used,
\begin{align*}
	\lambda(u) &= 40.1 +0.05u -10^{-4} u^2 + 4.9 \times 10^{-8} u^3, \\
	c_p(u) &= -10 \ln\left( \frac{\exp(-c_{p1}(u)/10) + \exp(-c_{p2}(u)/10)}{2}\right), 
\end{align*} where $c_{p1}(u)$ and $c_{p2}(u)$ are given by \begin{align*}
	\begin{aligned}
		c_{p1}(u) &= 34.2 \exp(0.0026 u) + 421.15,\\
		c_{p2}(u) &= 956.5 \exp(-0.012(u-900)) + 0.45u.
	\end{aligned}
\end{align*} The heat equation also uses zero Neumann boundary conditions on all non-coupling boundaries. The initial temperature is set to $900 K$.

As a discretization we use a linear finite element method in space and implicit Euler in time, implemented in the FEniCS software package \cite{aln2015} and based on the preCICE flow over heated plate tutorial \cite{chourdakis2022}. We use a uniform mesh in the solid solver and a fine discretization with $\Delta y = 1.5 \times 10^{-4}$ and $\Delta x = 2 \times 10^{-3}$.

At the coupling interface $(0, 0.2)\times\{0\}$, we impose continuity of temperature and normal heat flux. A DNWR method is used, with Dirichlet conditions on the fluid side and Neumann conditions on the solid side. The coupling is implemented using the open-source library preCICE \cite{chourdakis2022}, with nearest-neighbor interpolation for data transfer between the fluid and solid meshes. We also use preCICE's absolute convergence measure with a tolerance of $10^{-7}$ as the termination condition. 

\begin{figure}
	\centering
	\mbox{		
		\includegraphics[width=0.3\textwidth]{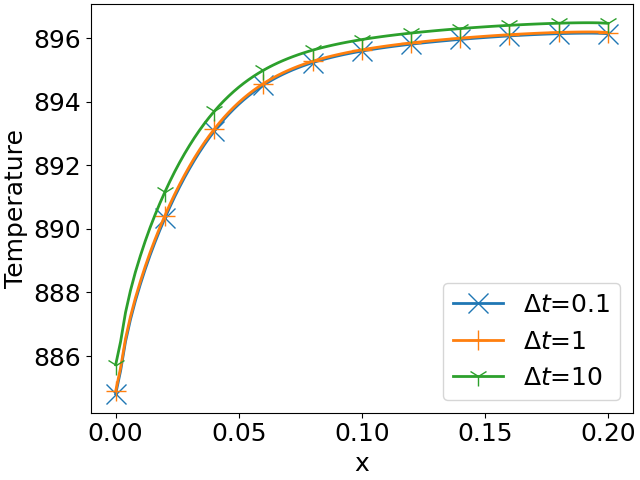}
	}
	\caption{The solution on the interface at $T=100$ for the different time discretizations.}
	\label{fig:flowOverPlateSol}
\end{figure}
We now run DNWR for $T=100$ and $\Delta t = 10,1,1,0.01$. Figure \ref{meshes} shows the solution at $T=100$ for $\Delta t = 1$ in the bottom panel and Figure \ref{fig:flowOverPlateSol} shows the solutions at the interface for the different time steps. The largest temperature decrease occurs at $x = 0$, which is consistent with the direction of the cold air stream. We observe that the solution for $\Delta t=10$ differs by about $2$K, suggesting a large time discretization error.

To make use of our analysis, we define a linear model by linearizing both nonlinear models around $900 K$, since we assume the temperature on the interface to have the largest effect on the heat transfer. The material parameters for the linearized steel model become $\lambda_{steel} = \lambda(900) = 40 $, $\alpha_{steel} = \alpha(900) = 6.1 \times 10^{6}$. For the air model we use the linearized parameters $\alpha_{air} =  \rho c_p = 1.293\times1005 = 1299.5$ and $\lambda_{air} = c_p \mu(900)/Pr = 0.057$ \cite{blazek2005}, where $c_p$ denotes the specific heat constant and $Pr$ the Prandtl number. In addition, we use domain lengths of $0.1$ and $0.015$ for air and steel, respectively. The optimal relaxation parameter then becomes $\theta^* = 0.9994$. 

Next, we compare the classical choice $\theta=1/2$ to $\theta^*$. In Figure \ref{fig:flowOverPlateIters}, we show the normalized error \[\norm{{\bf u}^k-{\bf u}^*}_{l_2}/\norm{{\bf u}^0-{\bf u}^*}_{l_2},\] where ${\bf u}^*$ denotes the interface temperature of a solution with an absolute residual less than $10^{-6}$ in preCICE. We also include the continuous and time-discrete error bounds~\eqref{eq:GenEstCon} and \eqref{eq:GenEstDis} for the linearized model for $\theta = 0.5$ and $\theta^*$. Here, the min-max problem in the error bound is solved numerically, using the Nelder-Mead optimization algorithm in {\tt SciPy} with an absolute tolerance of $10^{-6}$. 

\begin{figure}
	\centering
	\mbox{\includegraphics[width=0.4\textwidth]{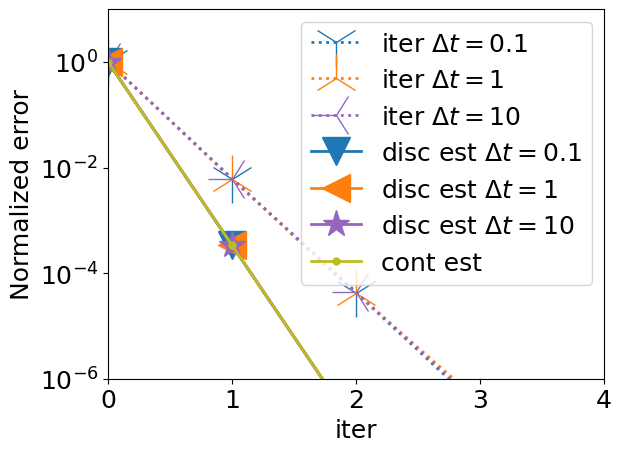}
	\includegraphics[width=0.4\textwidth]{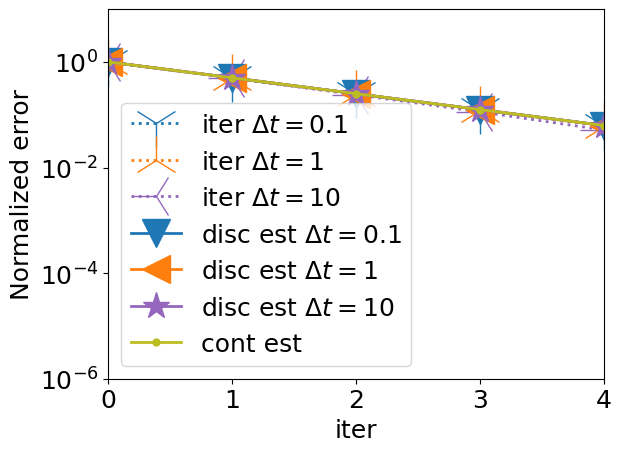}}
	\caption{Normalized error of the iteration together with the time-discrete and continuous error bounds in \eqref{eq:GenEstCon} and \eqref{eq:GenEstDis}. Left: $\theta = \theta^*$. Right: $\theta = 0.5$.}
	\label{fig:flowOverPlateIters}
\end{figure}

We see that the continuous and time-discrete error estimates predict the same linear convergence, independent of $\Delta t$. For $\theta=\theta^*$, the fast linear convergence and its independence on $\Delta t$ is explained by Corollaries \ref{cor:klarge} and \ref{cor:cont}. For $\theta=0.5$ the estimate predicts slow linear convergence, reaching an error of $\approx 10^{-2}$ in six iterations.

The normalized error also exhibits fast linear convergence for $\theta=\theta^*$ and we observe no dependence on $\Delta t$. Furthermore, the fully discrete iterates exhibit the same slow linear convergence for $\theta = 0.5$ as the error estimates \eqref{eq:GenEstCon} and \eqref{eq:GenEstDis}. This shows that $\theta=\theta^*$ yields fast convergence in advanced nonlinear 2D thermal cases, demonstrating the usefulness of our analysis and showing the importance of taking the material parameters into account in the choice of relaxation parameter. 




%% file: ourbibliography.bib
@article{arnoult2023,
	title = {Discrete-Time Analysis of Optimized {S}chwarz Waveform Relaxation with {R}obin Parameters Depending on the Targeted Iteration Count},
	author = {Arnoult, Arthur and Japhet, Caroline and Omnes, Pascal},
	year = {2023},
	journal = {ESAIM: M2AN},
	volume = {57},
	number = {4},
	pages = {2371--2396},
	issn = {2822-7840, 2804-7214},
	doi = {10.1051/m2an/2023051}}

@inproceedings{DD28,
	title = {Continuous Analysis of Waveform Relaxation for Heterogeneous Heat Equations},
	author = {Birken, Philipp and Gander, Martin J and Kotarsky, Niklas},
	editor = {P. Bjørstad and X.-C. Cai and V. Dolean and D. E. Keyes and R. Kornhuber and J. Xu},
	booktitle={Domain Decomposition Methods in Science and Engineering XXVIII},
	year={2025},
	volume = {155},
	pages={279--287}}

@article{aln2015,
	title = {The {{FEniCS Project Version}} 1.5},
	author = {Aln, Martin S and Kehlet, Benjamin and Logg, Anders and Richardson, Chris and Ring, Johannes and Rognes, E and Wells, Garth N},
	journal = {Archive of Numerical Software},
	volume = {3},
	number = {100},
	pages = {9--23},
	year = {2015}
}

@article{quint2011,
	title = {Experimental Validation of High-Order Time Integration for Non-Linear Heat Transfer Problems},
	author = {Quint, Karsten J. and Hartmann, Stefan and Rothe, Steffen and Saba, Nicolas and Steinhoff, Kurt},
	year = 2011,
	month = jul,
	journal = {Comput. Mech.},
	volume = {48},
	number = {1},
	pages = {81--96},
	issn = {0178-7675, 1432-0924}
}

@book{Ciaramella2022,
title = {Iterative Methods and Preconditioners for Systems of Linear Equations},
author = {Gabriele Ciaramella and Martin J. Gander},
year = {2022},
doi = {10.1137/1.9781611976908},
publisher = {SIAM},
address = {Philadelphia}}

@article{clement2022,
	title = {Discrete Analysis of {S}chwarz Waveform Relaxation for a Diffusion Reaction Problem with Discontinuous Coefficients},
	author = {Clement, S. and Lemari{\'e}, F. and Blayo, E.},
	year = {2022},
	journal = {SMAI J. Comput. Math.},
	volume = {8},
	pages = {99--124},
	issn = {2426-8399},
	doi = {10.5802/smai-jcm.81}}

@article{engstrom2024,
	title = {Linearly Convergent Nonoverlapping Domain Decomposition Methods for Quasilinear Parabolic Equations},
	author = {Engstr{\"o}m, Emil and Hansen, Eskil},
	year = {2024},
	journal = {BIT Numer. Math.},
	volume = {64},
	number = {37},
	doi = {10.1007/s10543-024-01038-5}}

@article{gander2007,
	title = {Optimized {S}chwarz Waveform Relaxation Methods for Advection Reaction Diffusion Problems},
	author = {Gander, M. J. and Halpern, L.},
	year = {2007},
	journal = {SIAM J. Numer. Anal.},
	volume = {45},
	number = {2},
	pages = {666--697},
	issn = {0036-1429, 1095-7170},
	doi = {10.1137/050642137}}

@article{gander,
	author = {Gander, Martin J.},
	journal = {ETNA},
	pages = {228-255},
	publisher = {Kent State University, Department of Mathematics and Computer Science},
	title = {Schwarz methods over the course of time.},
	url = {http://eudml.org/doc/130616},
	volume = {31},
	year = {2008}}

@incollection{blazek2005,
	title = {Introduction},
	booktitle = {Computational {{Fluid Dynamics}}: {{Principles}} and {{Applications}}},
	author = {Blazek, J.},
	year = 2005,
	edition = {Second Edition},
	pages = {1--4},
	publisher = {Elsevier},
	doi = {10.1016/B978-008044506-9/50003-5},
	urldate = {2026-08-10},
	copyright = {https://www.elsevier.com/tdm/userlicense/1.0/},
	isbn = {978-0-08-044506-9},
}

@article{chourdakis2022,
	title = {{{preCICE}} v2: {{A}} Sustainable and User-Friendly Coupling  Library},
	author = {Chourdakis, Gerasimos and Davis, Kyle and Rodenberg, Benjamin and Schulte, Miriam and Simonis, Fr{\'e}d{\'e}ric and Uekermann, Benjamin and Abrams, Georg and Bungartz, Hans-Joachim and Yau, Lucia Cheung and Desai, Ishaan and Eder, Konrad and Hertrich, Richard and Lindner, Florian and Rusch, Alexander and Sashko, Dmytro and Schneider, David and Totounferoush, Amin and Volland, Dominik and Vollmer, Peter and Koseomur, Oguz Ziya},
	year = 2022,
	journal = {Open Res. Eur.},
	volume = {2},
	number = {51},
	doi = {10.12688/openreseurope.14445.2}
}

@article{vynnycky1998,
	title = {Forced Convection Heat Transfer from a Flat Plate: The Conjugate Problem},
	shorttitle = {Forced Convection Heat Transfer from a Flat Plate},
	author = {Vynnycky, M. and Kimura, S. and Kanev, K. and Pop, I.},
	year = 1998,
	month = jan,
	journal = {Int. J. Heat Mass Transf.},
	volume = {41},
	number = {1},
	pages = {45--59},
	issn = {00179310},
	doi = {10.1016/S0017-9310(97)00113-0}
}

@article{weller1998,
	title = {A Tensorial Approach to Computational Continuum Mechanics Using Object-Oriented Techniques},
	author = {Weller, H. G. and Tabor, G. and Jasak, H. and Fureby, C.},
	year = 1998,
	month = nov,
	journal = {Comput. Phys.},
	volume = {12},
	number = {6},
	pages = {620--631},
	issn = {0894-1866},
	doi = {10.1063/1.168744},
}

@article{gakwma:16,
author = {Gander, M. J. and Kwok, F. and Mandal, B. C.},
issn = {10689613},
journal = {ETNA},
pages = {424--456},
title = {{Dirichlet-{N}eumann and {N}eumann-{N}eumann waveform relaxation algorithms for parabolic problems}},
volume = {45},
year = {2016}}

@book{waveformBook,
	title = {Time Parallel Time Integration},
	author = {Gander, Martin J. and Lunet, Thibaut},
	year = {2024},
	eprint = {https://epubs.siam.org/doi/pdf/10.1137/1.9781611978025},
	publisher = {SIAM},
	address = {Philadelphia},
	doi = {10.1137/1.9781611978025}}

@book{grafakos2014a,
	title = {Classical Fourier Analysis},
	author = {Grafakos, Loukas},
	year = {2014},
	series = {Graduate Texts in Mathematics},
	volume = {249},
	publisher = {Springer New York},
	doi = {10.1007/978-1-4939-1194-3},
	isbn = {978-1-4939-1193-6 978-1-4939-1194-3}}

@inproceedings{haynes2020,
	title = {Fully Discrete {S}chwarz Waveform Relaxation on Two Bounded Overlapping Subdomains},
	booktitle = {Domain {{Decomposition Methods}} in {{Science}} and {{Engineering XXV}}},
	author = {Haynes, Ronald D. and Mohammad, Khaled},
	editor = {Haynes, Ronald and MacLachlan, Scott and Xiao-Chuan Cai and Laurence Halpern and Hyea Hyun Kim and Axel Klawonn and Olof Widlund},
	year = {2020},
	volume = {138},
	pages = {159--166},
	publisher = {Springer International Publishing},
	address = {Cham},
	doi = {10.1007/978-3-030-56750-7_17}}

@article{hencha:09,
author = {Henshaw, W. D. and Chand, K. K.},
doi = {10.1016/j.jcp.2009.02.007},
issn = {00219991},
journal = {J. Comp. Phys.},
number = {10},
pages = {3708--3741},
publisher = {Elsevier Inc.},
title = {A composite grid solver for conjugate heat transfer in fluid–structure systems},
volume = {228},
year = {2009}}

@article{keyes2013,
	title = {Multiphysics Simulations: {{Challenges}} and Opportunities},
	shorttitle = {Multiphysics Simulations},
	author = {Keyes, David E and others},
	year = {2013},
	journal = {Int. J. High Perform. Comput. Appl.},
	volume = {27},
	number = {1},
	pages = {4--83},
	issn = {1094-3420, 1741-2846},
	doi = {10.1177/1094342012468181}}

@misc{code,
	type = {Github repository},
	author = {Kotarsky, Niklas},
	title =  {},
	year = {2026},
	howpublished = {https://github.com/NiklasKotarsky}}

@inproceedings{kowollik2013,
	title = {Fluid-Structure Interaction Analysis Applied to Thermal Barrier Coated Cooled Rocket Thrust Chambers with Subsequent Local Investigation of Delamination Phenomena},
	booktitle = {Progress in Propulsion Physics},
	author = {Kowollik, D. S. C. and Horst, P. and Haupt, M. C.},
	year = 2013,
	pages = {617--636},
	publisher = {EDP Sciences},
	address = {St. Petersburg, Russian},
	doi = {10.1051/eucass/201304617}}

@article{lelarasmee1982,
	title = {The Waveform Relaxation Method for Time-Domain Analysis of Large Scale Integrated Circuits},
	author = {Lelarasmee, E. and Ruehli, A.E. and {Sangiovanni-Vincentelli}, A.L.},
	year = {1982},
	journal = {IEEE Trans. Comput.-Aided Des. Integr. Circuits Syst.},
	volume = {1},
	number = {3},
	pages = {131--145},
	issn = {0278-0070},
	doi = {10.1109/TCAD.1982.1270004}}

@article{ledebl:13,
   author = {F. Lemarié and L. Debreu and E. Blayo},
   journal = {ETNA},
   pages = {148-169},
   title = {Toward an optimized global-in-time {S}chwarz algorithm for diffusion equations with discontinuous and spatially variable coefficients. {P}art 1: the constant coefficients case},
   volume = {40},
   year = {2013}}

@book{Lions72,
	title = {Non-Homogeneous Boundary Value Problems and Applications II},
	author = {Lions, J. L. and Magenes, E.},
	year = 1972,
	volume = {182},
	publisher = {Springer Berlin},
	isbn = {3-540-05444-8}}

@article{lorentzon2020,
	title = {A Numerical Study of Partitioned Fluid-structure Interaction Applied to a Cantilever in Incompressible Turbulent Flow},
	author = {Lorentzon, Johan and Revstedt, Johan},
	year = {2020},
	journal = {Int. J. Numer. Methods Eng.},
	volume = {121},
	number = {5},
	pages = {806--827},
	issn = {0029-5981, 1097-0207},
	doi = {10.1002/nme.6245}}

@InProceedings{ma:13,
	author="Mandal, Bankim C.",
	editor="Erhel, Jocelyne and Martin J. Gander and Laurence Halpern and Géraldine Pichot and Taoufik Sassi and Olof Widlund",
	title="A Time-Dependent {D}irichlet-{N}eumann Method for the Heat Equation",
	booktitle="Domain Decomposition Methods in Science and Engineering XXI",
	year="2014",
	publisher="Springer International Publishing",
	address="Cham",
	pages="467--475",
	volume = "98",
	isbn="978-3-319-05789-7"}

@article{meimonbir:23,
	author = {P. Meisrimel and A. Monge and P. Birken},
	issue = {11},
	journal = {ZAMM},
	title = {A time adaptive multirate {D}irichlet–{N}eumann waveform relaxation method for heterogeneous coupled heat equations},
	volume = {103},
	url = {https://onlinelibrary.wiley.com/doi/pdf/10.1002/zamm.202100328},
	year = {2023}}

@misc{Me25,
	type = {Github Repository},
	author = {Meisrimel, Peter},
	title =  {},
	year = {2026},
	howpublished = {https://github.com/PeterMeisrimel/DNWR}}

@article{monbir:18,
author = {Monge, Azahar and Birken, Philipp},
doi = {10.1007/s00466-017-1511-3},
issn = {01787675},
journal = {Comput. Mech.},
number = {3},
pages = {525--541},
publisher = {Springer Berlin Heidelberg},
title = {On the convergence rate of the {D}irichlet–{N}eumann iteration for unsteady thermal fluid–structure interaction},
volume = {62},
year = {2018}}

@book{Oppenheim1999,
	title = {Discrete-Time Signal Processing},
	edition = {2nd},
	author = {Alan V. Oppenheim and Ronald W. Schafer},
	year = {1999},
	publisher = {Prentice-Hall, Inc.}}

@article{rodenberg2025,
	title={A waveform iteration implementation for black-box multi-rate higher-order coupling}, 
	author={Benjamin Rodenberg and Benjamin Uekermann},
	year={2025},
	journal = {Submitted},
	volume = {},
	pages = {},
	number = {}}

@article{Kwok2025,
	title = {Convergence analysis of optimized {S}chwarz waveform relaxation by exponential weighting},
	journal = {Submitted},
	year = {2026},
	author = {Alejandro Alfonso Rodriguez and Felix Kwok and Martin J. Gander},
	volume = {},
	number = {}}

@article{schuller2025,
	title = {Quantifying Coupling Errors in Atmosphere-Ocean-Sea Ice Models: {A} Study of Iterative and Non-Iterative Approaches in the {EC-Earth AOSCM}},
	author = {Sch{\"u}ller, Valentina and Lemari{\'e}, Florian and Birken, Philipp and Blayo, Eric},
	year = {2025},
	journal = {Geosci. Model Dev.},
	volume = {18},
	number = {22},
	pages = {9167--9187},
	issn = {1991-9603},
	doi = {10.5194/gmd-18-9167-2025}}

@book{birken:21,
author = {Birken, P.},
pages = {246},
publisher = {CRC Press},
title = {{Numerical Methods for Unsteady Compressible Flow Problems}},
year = {2021}
}
